\documentclass[10pt]{IEEEtran}

\usepackage{graphicx}
\usepackage{tikz}
\usepackage{pgfplots}
\pgfplotsset{compat=1.5}

\usetikzlibrary{calc}
\usetikzlibrary{spy}

\usepackage{amsfonts}
\usepackage{amsmath}
\usepackage{amsthm}
\usepackage{mathtools}
\usepackage{calrsfs}
\usepackage{manfnt}

\usepackage[mathscr]{euscript}

\usepackage[noadjust]{cite}

\theoremstyle{definition}

\newtheorem{theorem}{Theorem}

\newtheorem{corollary}{Corollary}

\newtheorem{proposition}{Proposition}
\newtheorem{example}{Example}

  \title{Fekete points, formation control, \\ and the balancing problem}
\author{Jan Maximilian Montenbruck, Daniel Zelazo, and Frank Allg\"ower
\thanks{
JM Montenbruck and F Allg\"ower are with the Institute for Systems Theory and Automatic Control, University of Stuttgart, and thank the German Research Foundation (DFG) for financial support of the project within the Cluster of Excellence in Simulation Technology (EXC 310/2) at the University of Stuttgart. D Zelazo is with the Faculty of Aerospace Engineering, Technion (Israel Institute of Technology). All authors were supported by the German-Israeli Foundation for Scientific Research and Development. For correspondence, \tt\fontsize{7.78pt}{1em}\selectfont mailto:jan-maximilian.montenbruck@ist.uni-stuttgart.de}
}

\begin{document}
\maketitle
\thispagestyle{empty}
\pagestyle{empty}
\begin{abstract}
We study formation control problems. Our approach is to let a group of systems maximize their pairwise distances whilst bringing them all to a given submanifold, determining the shape of the formation. 
The algorithm we propose allows to initialize the positions of the individual systems in the ambient space of the given submanifold but brings them to the desired formation asymptotically in a stable fashion. Our control inherently consists of a distributed component, maximizing the pairwise distances, and a decentralized component, asymptotically stabilizing the submanifold. We establish a graph-theoretical interpretation of the equilibria that our control enforces and extend our approach to systems living on the special Euclidean group. Throughout the paper, we illustrate our approach on different examples.
\end{abstract}

\section{Introduction}
Multi-agent systems have become one of the central foci of attention in control theory. %
This interest partially stems from the relevance of related methods for control of robotic networks, cf. \cite{Bullo2009}. The central question in these systems usually reads as follows: which control algorithms will eventually drive the  group of systems to a desired configuration? The desired configuration itself will thereby depend on the particular group objective under scrutiny. For instance, one often wishes to have the systems eventually arrange their positions in a given shape or pattern. The task is trivial if the individual systems are controllable and one allows for controllers which drive them to pre-computed positions within the chosen formation shape; yet, if one was to solve the problem in this fashion, one would require a central processing entity gathering all information and sending commands to all members of the group. At the same time, a new controller would have to be designed 
whenever the formation objectives changed, limiting its applicability. Instead, 
it would be more desirable to have the systems automatically arrange in the desired formation while exchanging only relative (``distributed control'') or individual (``decentralized control'') information, cf. \cite{Fax2004}
.

In this paper, we study precisely these formation control problems, i.e., tasks in which a group of systems is asked to eventually arrange their positions in a specified shape. This shape shall thereby be defined by a compactly embedded submanifold of the space which the systems live in. This 
\phantom{Neque porro quisquam est qui dolorem ipsum quia dolor sit amet}
approach offers for great flexibility in the formation shape, quite similar to \cite{Zhang2007}, wherein the formation may be determined by an arbitrary Jordan curve. By construction, we arrive at a control law consisting of a distributed and a decentralized component.

In the following section, we link the formation control problem to the problem of asymptotically stabilizing so-called \emph{Fekete points}. Thereafter, we compare this approach with existing approaches to formation control. In section \ref{sec:control}, we present a control law that is shown to asymptotically stabilize these Fekete points. This control law is illustrated on circular and spherical formation shapes in section \ref{sec:circlsph}. Then, in section \ref{sec:graph}, we establish a novel connection between Fekete points and cycle spaces of graphs, i.e., that any equilibrium configuration of our control law must correspond to elementwise reciprocals of vectors from the cycle space of the underlying communication graph. The approach pursued in section \ref{sec:Eucl} equips our control with the capability to take orientations into account, thus allowing us to stabilize formations in the Euclidean groups. This, again, is illustrated on circular and spherical formations in section \ref{sec:circlsphEucl}. In section \ref{sec:ext}, we point towards a number of further possible extensions and section \ref{sec:conc} concludes the paper.

\section{Fekete Points and the Balancing Problem}
Unlike other approaches to formation control, we do not (implicitly) define the desired formation by specifying absolute or relative positions of systems. Rather, our notion of a formation corresponds to when the systems in the group arrange their positions in a \emph{balanced} fashion on some pre-specified shape. More formally, we ask for $n$ systems to arrange their positions $x_{1},\dots,x_{n}$ according to the shape of a given compactly embedded submanifold $M$ of $\mathbb{R}^{m}$. 
By ``balanced'', we mean the positions of the agents should be evenly spaced in the submanifold. That is, we wish to avoid situations in which two positions $x_{i}$, $x_{j}$ are close to each other. 

In this direction, one may be tempted to think that the maximization of
\begin{equation}
 \sum_{j\,>\,i}d\left(x_{i},x_{j}\right)^{2} \label{eq:oldcost}
\end{equation}
subject to $x_{i}\in M$, wherein $d\left(x_{i},x_{j}\right)$ is the length of the shortest curve (in $M$) joining $x_{i}$ and $x_{j}$, endowing $M$ with the properties of a metric space, yields such configurations. In the following example, we briefly illustrate why this approach is flawed.
\begin{example}
Let $M$ be the unit circle in $\mathbb{R}^{2}$ and consider the $n=3$ points
\begin{equation}
x_{1}=x_{2},\;\;\;x_{3}=-x_{1}, \label{eq:excircconf1}
\end{equation}
for which our function \eqref{eq:oldcost} attains the value $2\pi^{2}$. A more desirable, ``balanced'' {(also, ``splay'')}, configuration, however, would correspond to the positions 
\begin{align}
x_{2}&=\begin{bmatrix}\operatorname{cos}\left(2\pi/3\right) & -\operatorname{sin}\left(2\pi/3\right) \\ \operatorname{sin}\left(2\pi/3\right) & \hphantom{-}\operatorname{cos}\left(2\pi/3\right)\end{bmatrix}x_{1}, \label{eq:excircconf2a} \\
x_{3}&=\begin{bmatrix}\operatorname{cos}\left(2\pi/3\right) & -\operatorname{sin}\left(2\pi/3\right) \\ \operatorname{sin}\left(2\pi/3\right) & \hphantom{-}\operatorname{cos}\left(2\pi/3\right)\end{bmatrix}x_{2}, \label{eq:excircconf2b}
\end{align}
which yields the smaller value $4\pi^{2}/3$ for \eqref{eq:oldcost}. Two exemplary configurations sufficing \eqref{eq:excircconf1} and \eqref{eq:excircconf2a}-\eqref{eq:excircconf2b} are depicted left and right in Fig. \ref{fig:excircconfs}, respectively. 

\begin{figure}\centering
 \begin{tikzpicture}[scale=1.0]
 \draw (0,0) circle (1);
 \filldraw[red] (1,0) circle (.05) node[anchor=east,xshift=-.5,yshift=-.5] {$x_{3}$};
 \filldraw[red] (-1,0) circle (.05) node[anchor=west,xshift=.5,yshift=-.5] {$x_{1},x_{2}$};
 \draw[<->] (1.1,0) arc (0:180:1.1);
 \draw[white] (0,1.1) -- (0,1.1) node[anchor=south] {\color{black}$d\left(x_{1},x_{3}\right)=\pi$}; 
 \end{tikzpicture}\begin{tikzpicture}[scale=1.0]
  \draw[white] (-2.43,1.675) -- (-2.43,1.675);
 \draw (0,0) circle (1);
 \filldraw[red] (1,0) circle (.05) node[anchor=east,xshift=-.5,yshift=-.5] {$x_{1}$};
 \filldraw[red] (-.5,{+sqrt(3)/2}) circle (.05) node[anchor=north west,yshift={-.5*sqrt(3)/2},xshift=.25] {$x_{2}$};
 \filldraw[red] (-.5,{-sqrt(3)/2}) circle (.05) node[anchor=south west,yshift={.5*sqrt(3)/2},xshift=.25] {$x_{3}$};
 \draw[<->] (1.1,0) arc (0:120:1.1);
  \draw[white] (.3,1.1) -- (.3,1.1) node[anchor=south] {\color{black}$d\left(x_{1},x_{2}\right)=2\pi/3$};
 \end{tikzpicture}
\caption{Configurations on the circle satisfying \eqref{eq:excircconf1} (left) and \eqref{eq:excircconf2a}-\eqref{eq:excircconf2b} (right).}\label{fig:excircconfs}
\end{figure}
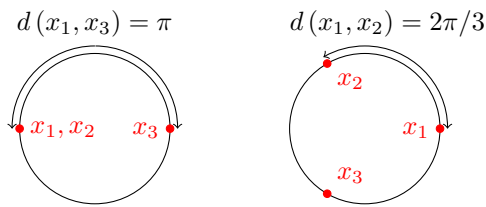

\end{example}	

To compensate for the shortcomings of \eqref{eq:oldcost}, let us seek for a function which attains very small values as any two positions $x_{i}$ and $x_{j}$ approach each other. To this end, consider
\begin{equation}
 \sum_{j\,>\,i}\operatorname{ln}\left(d\left(x_{i},x_{j}\right)\right)=\operatorname{ln}\left(\vphantom{\prod}\right.\!\prod_{j\,>\,i}d\left(x_{i},x_{j}\right)\!\!\left.\vphantom{\prod}\right), \label{eq:newcost}
\end{equation}
which now tends to $-\infty$ as any pairwise distinct points $x_{i}$, $x_{j}$ approach each other. Yet, due to strict monotonicity of the natural logarithm, we expect that the values attained for configurations such as \eqref{eq:excircconf2a}-\eqref{eq:excircconf2b} are still large. 

It turns out that the cost function \eqref{eq:newcost} is not new to the exact sciences. For the special case of $M$ being the sphere, its maximizers are today referred to as (elliptic) \emph{Fekete points}. Thomson asked for these points while studying electronically charged particles, subject to Coulomb's law, constrained to the sphere \cite{Thomson1904}. This problem was then brought to mathematics by F\"oppl \cite{Foeppl1912} on advice of his advisor Hilbert. Later, Schur \cite{Schur1918} asked for polynomials with large discriminant and roots in the unit interval or, similarly, for large values of the Vandermonde polynomials with arguments in the unit interval, leading Fekete \cite{Fekete1923} to ask the same question for these roots / arguments constrained to arbitrary compact sets and eventually giving these points their present name. More recently, Shub and Smale \cite{Shub1993} required Fekete points as initial conditions for an algorithm {computing} zeros of (complex) homogeneous polynomials, letting Smale define their (algorithmic, in the sense of Blum-Cucker-Shub-Smale) computation as one of \emph{the} mathematical problems of our century \cite{Smale1998}.

Returning to our problem, we thus ask for our positions $x_{1},\dots,x_{n}$ to eventually attain such Fekete points, and if possible, in a stable fashion. Although this point of view on the formation control problem is, to our best knowledge, novel, others have presented conceptually similar definitions of desirable configurations. In sensor coverage, one steers systems to centroids of a Voronoi diagram through a continuous-time version of Lloyd's algorithm \cite{Cortes2004}. Formation shapes may thereby be taken into account via specific density functions. Circular formations can be stabilized by minimizing all angular moments \cite{Sepulchre2007} or by zeroing their centroid \cite{Scardovi2007}, which is in this context often referred to as \emph{balancing}
. Formations whose shape is determined by a Jordan curve can be stabilized by choosing the desired relative distances a priori \cite{Zhang2007}. If the formation shape is a more general (homogeneous) manifold, it may still be stabilized by maximizing the pairwise chordal distances of the individual systems \cite{Sarlette2009}. The dual consensus problem has also been solved intrinsically \cite{Tron2013,Montenbruck2015a}. 
Minimizing the deviation of relative distances among agents from the lengths of the links in a rigid framework stabilizes the formation defined by that framework \cite{Olfati-Saber2002}. The weaker notion of infinitesimal rigidity proves to be sufficient for this purpose, as well \cite{Krick2009}.

A significant distinction between \cite{Sarlette2009} and \cite{Zhang2007} is that the former does not asymptotically drive the positions to the specified manifold but expects that the positions are constrained to the manifold for all times while the latter expects that the desired relative distances $d\left(x_{i},x_{j}\right)$ are specified a priori (the former does not assume to know these relative distances and the latter does allow for the positions to move in the ambient space $\mathbb{R}^{2}$ of the chosen Jordan curve). In the present paper, we allow for our positions $x_{1},\dots ,x_{n}$ to move in the ambient space $\mathbb{R}^{m}$ of some compactly embedded smooth submanifold $M$ and impose no prespecification of desired relative distances. Instead, we let the maximizers of \eqref{eq:newcost}, our Fekete points, specify the desired configurations on $M$, a point of view which is, to our knowledge, novel.

{We thus introduce and study Fekete points as a natural definition of evenly spaced formations. {This is in contrast to defining such an even spacing as a configuration with zero centroid}, all pairwise relative distances equal, or all polytopes connecting nearby points being of the same type. For instance, with $M$ being the sphere in $\mathbb{R}^{3}$, {$n=11$ yields Fekete points whose centroid is not at the origin (topologically equivalent to an edge-contracted icosahedron, left in Fig. \ref{fig:polytopes})}, $n=5$ yields Fekete points whose relative distances are not all the same (topologically equivalent to a triangular bipyramid, middle in Fig. \ref{fig:polytopes}), and $n=8$ yields Fekete points connected by both quadrilateral and triangular polygons (topologically equivalent to a square antiprism, right in Fig. \ref{fig:polytopes}).

\begin{figure}
\includegraphics[width=.16\textwidth]{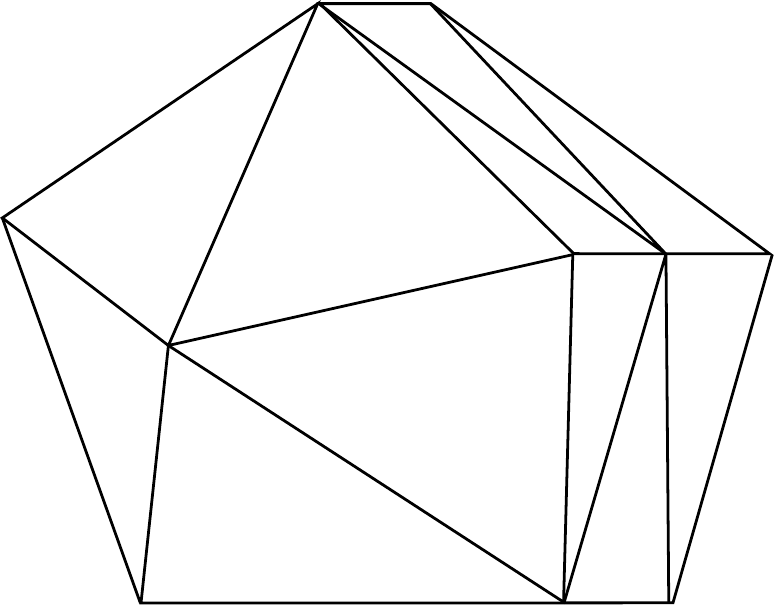}\hspace{1.5pt}\raisebox{-2pt}{\includegraphics[width=.16\textwidth]{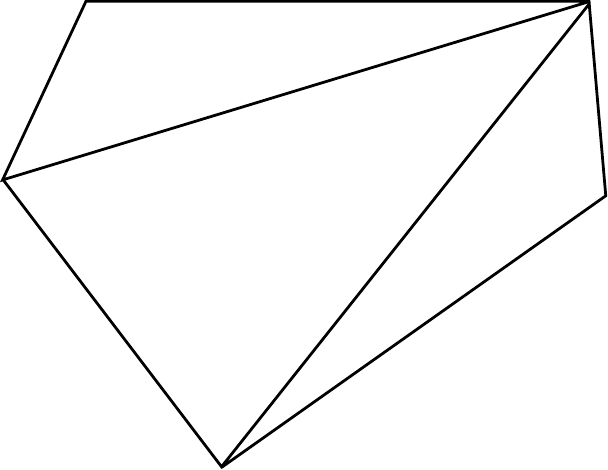}}\hspace{2pt}\includegraphics[width=.16\textwidth]{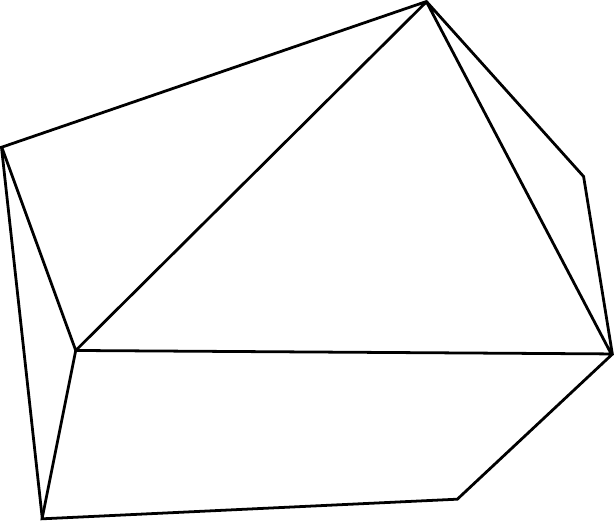}
 \caption{Polytopes constituted by Fekete points on the sphere: edge-contracted icosahedron for $n=11$ (left), triangular bipyramid for $n=5$ (middle), and square antiprism for $n=8$ (right).}\label{fig:polytopes}
\end{figure}

}

\section{Asymptotic Stability of Fekete Points}\label{sec:control}
Let $x_{1},\dots,x_{n}$ denote the positions of our systems in $\mathbb{R}^{m}$ and let $M$ be a smooth, compactly embedded submanifold of $\mathbb{R}^{m}$. We seek algorithms which drive our positions towards maximizers of \eqref{eq:newcost}, subject to $x_{i}\in M$, in a stable fashion. For greater flexibility, we enhance the expression \eqref{eq:newcost} with scalar nonnegative weights $W_{ij}$, i.e., we consider the function
\begin{equation}
 \phi\left(x\right)=\sum_{j\,>\,i}W_{ij}\operatorname{ln}\left(d\left(x_{i},x_{j}\right)\right), \label{eq:phidef}
\end{equation}
taking members $x:=\left(x_{1},\dots,x_{n}\right)$ of the product manifold $M^{n}=M\times\cdots\times M$ (excluding the points $\Delta\subset M^{n}$ for which any projection $x\mapsto\left(x_{i},x_{j}\right)$, $i\neq j$, lies in the diagonal of $M^{2}$, or, equivalently, $i\mapsto x_{i}$ is no injection{, since $\phi$ cannot be evaluated at these points}), to the real line, where, again, $d\left(x_{i},x_{j}\right)$ denotes the length of the shortest curve in $M$ {joining} $x_{i}$ and $x_{j}$. We can associate an undirected, weighted graph with $n$ vertices to the symmetric function $\left(i,j\right)\mapsto W_{ij}$ by letting $i$ and $j$ be neighbors only if $W_{ij}=W_{ji}$ is positive and by letting $W_{ij}$ be the weight of the edge which connects them if this is the case.

In order to proceed, we require some terminology. Let $$T_{x}\left(M\times\cdots\times M\right)=T_{x_{1}}M\oplus\cdots\oplus T_{x_{n}}M$$ denote the tangent space of $M^{n}$ at $x$ and let $N_{x}M^{n}$ be the normal space (in $\mathbb{R}^{mn}$) of $M^{n}$ at $x$, defined as the orthogonal complement of $T_{x}M^{n}$ in $\mathbb{R}^{mn}$. The real vector bundle $NM^{n}=\bigsqcup_{x\in M^{n}}N_{x}M^{n}$, composed of the fibers $N_{x}M^{n}$, is called the \emph{normal bundle} of $M^{n}$. A \emph{tubular neighborhood} of $M^{n}$ is a diffeomorphic image of $NM^{n}\to\mathbb{R}^{mn}$, $\left(x,v\right)\mapsto x+v$. The tubular neighborhood theorem asserts that embedded submanifolds have tubular neighborhoods. Moreover, following the construction in \cite[chapter II, section 11]{Bredon1993}, compact embedded submanifolds have tubular neighborhoods that are sublevel sets of $x\mapsto\left\|v\right\|$ (with the previously employed notation) and we will henceforth always refer to such. Now let $U\subset\mathbb{R}^{mn}$ be such a tubular neighborhood of $M^{n}$; then $r:x+v\mapsto x$ is a smooth retraction from $U$ onto $M$. Let $\operatorname{grad}\phi$ denote the gradient vector field of our scalar field $\phi:M^{n}\setminus\Delta\to\mathbb{R}$, 
i.e., $\operatorname{grad}\phi$ accepts arguments $x$ from $M^{n}\setminus\Delta$ (except for those points $\operatorname{Cut}\subset M^{n}$ at which any $x_{i}$ lies in the cut loci of some $x_{j}$, as $\phi$ is not differentiable there; yet, these points only constitute a set of measure zero anyhow) and takes them to vectors in $T_{x}M^{n}$. We propose the control
\begin{equation}
 \dot{x}=r\left(x\right)-x+\operatorname{grad}\phi\left(r\left(x\right)\right) \label{eq:ODE}
\end{equation}
in order to drive our positions $x$ from some initial condition $x_{0}\in r^{-1}\left(M^{n}\setminus\left(\Delta\cup\operatorname{Cut}\right)\right)$ towards maximizers of $\phi$ on $M$, our Fekete points, in a stable fashion {(needless to say, these maximizers will thereby also depend upon the choice of the weights $W_{ij}$)}. 

By construction, our control consists of a decentralized and a distributed component. The vector field $r$ can be computed in a decentralized fashion since the $i$th entry of $r\left(x\right)$ is just the retraction of $x_{i}$ onto $M$ -- the vector field $\operatorname{grad}\phi\circ r$ can be computed in a distributed fashion as, by the chain rule, its $i$th entry (evaluated at $x$) reads
\begin{equation}
 \sum_{j=1}^{n}\frac{W_{ij}}{d\left(r\left(x_{i}\right),r\left(x_{j}\right)\right)}\,V_{ij} \label{eq:ithODE}
\end{equation}
where $\left(x_{i},x_{j}\right)\mapsto V_{ij}\in T_{r\left(x_{i}\right)}M$ is the (initial) velocity vector of the unit speed geodesic joining $r\left(x_{i}\right)$ and $r\left(x_{j}\right)$. 

One can see that $\phi$ attains its maximum on $M^{n}$ as $$\operatorname{e}^{\phi\left(x\right)}=\prod_{j\,>\,i}d\left(x_{i},x_{j}\right)^{W_{ij}}$$ is continuous, $M^{n}$ is compact, and the natural logarithm is strictly monotone. In the remainder, we denote this maximum by $\phi^{\ast}$ (implying $\phi^{\ast}\geq\phi\left(x\right)$ for any $x\in M^{n}\setminus\Delta$) and the maximizers by $X^{\ast}:=\phi^{-1}\left(\lbrace\phi^{\ast}\rbrace\right)$. 

\begin{theorem}\label{thm:main}
Let $X$ be a superlevel set of $\phi$ on which $\phi$ is regular away from the maximizers $X^{\ast}$. These maximizers constitute an asymptotically stable set of equilibria of \eqref{eq:ODE} and $r^{-1}\left(X\right)$ is a subset of their region of asymptotic stability.
\end{theorem}
\begin{proof}
We prove our claim as follows: first, we show that $r^{-1}\left(X^{\ast}\right)$ is an asymptotically stable invariant set by evaluating the evolution of $r\left(x\right)$ along solutions of \eqref{eq:ODE}. Second, we study the differential equation under which $v\left(x\right):=x-r\left(x\right)$ evolves to find that $M^{n}$ is also an asymptotically stable invariant set. The proof concludes by recalling that intersections of asymptotically stable invariant sets are themselves asymptotically stable.

We first analyze how $r\left(x\right)$ evolves under \eqref{eq:ODE}. The Jacobian of $r$, evaluated at $x$, is just the projection matrix of the projection onto $T_{r\left(x\right)}M^{n}$. Keeping in mind that $r\left(x\right)-x$ is a vector from the normal space of $M^{n}$ at $x$, it follows that $r\left(x\right)$ obeys the differential equation
 \begin{equation}
  \dot{r}\left(x\right)=\operatorname{grad}\phi\left(r\left(x\right)\right). \label{eq:rODE}
 \end{equation}
As our maximizers are critical points of $\phi$, they are also equilibria of \eqref{eq:rODE}. We note that $\phi^{\ast}-\phi\left(x\right)$ is positive away from the maximizers per definition. Further, since $\phi$ is regular on $X\setminus X^{\ast}$, the Lie derivative of $-\phi$ along $\operatorname{grad}\phi$ is negative on that set, whence the maximizers are indeed an asymptotically stable set of equilibria of \eqref{eq:rODE} by virtue of Lyapunov's direct method. With the aforementioned Lie derivative being nonpositive on $X$ and $X$ being a sublevel set of $-\phi$, $X$ remains an invariant set of \eqref{eq:rODE}. As $M$ is compact, $X$ is compact and thus it belongs to the region of asymptotic stability of $X^{\ast}$ by LaSalle's invariance principle. It follows that $r^{-1}\left(X^{\ast}\right)$ is an asymptotically stable invariant set of \eqref{eq:ODE} whose region of attraction is at least $r^{-1}\left(X\right)$.

Now we turn our attention to the evolution of $v\left(x\right)$ along solutions of \eqref{eq:ODE}. As the Jacobian of $v$, evaluated at $x$, is just the projection matrix of the projection onto $N_{r\left(x\right)}M^{n}$ and $\operatorname{grad}\phi\left(r\left(x\right)\right)$ is always tangent to $M^{n}$ at $r\left(x\right)$, we come to the conclusion that $\dot{v}\left(x\right)=-v\left(x\right)$. Therefore, all solutions of \eqref{eq:ODE} initialized in the invariant set $r^{-1}\left(X\right)$ approach $v^{-1}\left(\lbrace 0\rbrace\right)=M^{n}$ asymptotically in a stable fashion, i.e., $M^{n}$ is an asymptotically stable invariant set of \eqref{eq:ODE} whose region of asymptotic stability is at least $r^{-1}\left(X\right)$  (cf. \cite[proof of Theorem 2]{Montenbruck2016}).

Bearing in mind that intersections of asymptotically stable invariant sets are asymptotically stable, we find that $r^{-1}\left(X^{\ast}\right)\cap M^{n}=X^{\ast}$ is an asymptotically stable set of equilibria of \eqref{eq:ODE}. Its region of asymptotic stability is at least the intersection of the two regions of asymptotic stability, that is $r^{-1}\left(X\right)$, completing the proof.
\end{proof}

In the proof, the preimages of $r$ appeared frequently. In particular, we were unable to extend the region of asymptotic stability of our maximizers beyond tubular neighborhoods. On a conceptual level, this agrees with the obstructions to global stabilization of certain formations observed in \cite{Belabbas2013}. 

{If we only ask our tubular neighborhood to be a diffeomorphic image of $NM^{n}\to\mathbb{R}^{mn}$, $\left(x,v\right)\mapsto x+v$, but not necessarily a sublevel set of $x\mapsto\left\|v\left(x\right)\right\|$, then it will be possible to also apply our control to positions $x$ outside those sublevel sets, but we would not be able to provide convergence guarantees for solutions initialized with such positions.  }

\section{Tutorial Examples: The Circle and the Sphere}\label{sec:circlsph}
Our control \eqref{eq:ODE} is rather general but also quite abstract. It is instructive to see how the involved expressions read for particular manifolds. In this section, we compute the right-hand side of \eqref{eq:ODE} explicitly for the circle (embedded in the plane) and for the sphere (embedded in $\mathbb{R}^{3}$).

Circular formations are among the most relevant formations in the plane $\mathbb{R}^{2}$ and have been extensively studied, e.g., in \cite{Sepulchre2007,Scardovi2007}. One reason for the relevance of circular formations is that they can be continuously deformed to other Jordan curves \cite{Zhang2007}, thus making methods which were initially developed for the circle applicable to a broad range of planar formations. In the following, we compute our control \eqref{eq:ODE} for $M$ being the (unit) circle and for $M$ being an ellipse, both embedded in the plane.
\begin{example}\label{ex:circ}
Let $M$ be the unit circle in $\mathbb{R}^{2}$. The retraction of some point $x_{i}$ from the tubular neighborhood of the circle on which $0<\left\|x_{i}\right\|<2$ is just the normalized vector 
\begin{equation}
 r\left(x_{i}\right)=\frac{1}{\left\|x_{i}\right\|}x_{i}. \label{eq:circretr}
\end{equation}
{Here, it is possible to retract any vector from the punctured plane $\mathbb{R}^{2}\setminus\lbrace 0\rbrace$ onto the circle, thus allowing us to also apply our control to positions $x_{i}$ outside our tubular neighborhood (though not having convergence guarantees for solutions initialized with such positions). }
It 
remains to compute the gradient of $\phi$. For this purpose, we employ the (Lie) group isomorphism
 $$\begin{bmatrix}\operatorname{cos}\left(\alpha\right) \\ \operatorname{sin}\left(\alpha\right)\end{bmatrix}\mapsto\begin{bmatrix}\operatorname{cos}\left(\alpha\right) & -\operatorname{sin}\left(\alpha\right) \\ \operatorname{sin}\left(\alpha\right) & \hphantom{-}\operatorname{cos}\left(\alpha\right)\end{bmatrix}$$
 from the circle onto the special orthogonal group $\operatorname{SO}\left(2\right)$. This representation is quite convenient as tangent vectors become skew-symmetric matrices which, in turn, become tangent vectors of the circle again by multiplying them with points on the circle (from the right). Specifically, geodesics on $\operatorname{SO}\left(2\right)$ (and their velocity vectors) can through this reasoning be employed to compute geodesics on the circle. Employing the notation $\left(x_{i},x_{j}\right)\mapsto V_{ij}$ from \eqref{eq:ithODE}, we find that 
 \begin{equation*}
  -d\left(x_{i},x_{j}\right)V_{ij}=\operatorname{log}\!\left(\frac{1}{\left\|x_{i}\right\|\!\left\|x_{j}\right\|}{\fontsize{10pt}{1em}\selectfont\begin{bmatrix} x_{i}\cdot x_{j}\! & \!x_{i}\cdot\Omega x_{j} \\ x_{j}\cdot\Omega x_{i}\! & \!x_{i}\cdot x_{j}\end{bmatrix}}\right)\!\frac{x_{i}}{\left\|x_{i}\right\|}
 \end{equation*}
 wherein ``$\cdot$'' denotes the scalar product and $\Omega$ is the infinitesimal generator
 \begin{equation}
  \Omega:=\begin{bmatrix}\hphantom{-}0 & 1 \\ -1 & 0\end{bmatrix}
 \end{equation}
of the Lie algebra $\mathfrak{so}\left(2\right)$ and $\operatorname{log}:\operatorname{SO}\left(2\right)\to\mathfrak{so}\left(2\right)$ is the logarithmic map. Dividing by $d\left(x_{i},x_{j}\right)$ twice can be efficiently realized by applying the identity $\Omega^{-1}=-\Omega$ and finally reveals that \eqref{eq:ithODE} reads
\begin{align}
 \dot{x}_{i}&=\left(\frac{1-\left\|x_{i}\right\|}{\left\|x_{i}\right\|}\right)x_{i} \label{eq:circODE} \\
 &\hspace{1pt}+\sum_{j=1}^{n}\frac{W_{ij}}{\left\|x_{i}\right\|}\!\left(\!\operatorname{log}\!\left(\frac{1}{\left\|x_{i}\right\|\!\left\|x_{j}\right\|}{\fontsize{10pt}{1em}\selectfont\begin{bmatrix}x_{i}\cdot x_{j}\! & \!x_{i}\cdot\Omega x_{j} \\ x_{j}\cdot\Omega x_{i}\! & \!x_{i}\cdot x_{j}\end{bmatrix}}\right)\!\right)^{-1}\!\!\!x_{i} \nonumber
\end{align}
for the present example. We now consider $n=10$ systems coupled through the unweighted cycle graph $C_{10}$, i.e., $W_{ij}=W_{ji}=1$ for $j=\left(i+1\right)\operatorname{mod}10$ and $W_{ij}=W_{ji}=0$ otherwise; the graph is depicted in Fig. \ref{fig:C10}. With this choice of graph, we solved \eqref{eq:circODE} numerically for some initial condition; the numerical solutions are plotted in Fig. \ref{fig:Efig1}. The initial condition is indicated by blue circles (\raisebox{1pt}{\tikz{\filldraw[blue] (0,0) circle (.06);}}) and the limiting point is marked with red circles (\raisebox{1pt}{\tikz{\filldraw[red] (0,0) circle (.06);}}). {Although the initial conditions where chosen outside our tubular neighborhood, w}e find that the positions approach an evenly spaced configuration on the circle, as desired.
\begin{figure}
\centering
\begin{tikzpicture}[scale=1.2]
\draw (0,0) circle (1);
 \draw[fill=white,rotate=0] (1,0) circle (.225) node {$1$};
  \draw[fill=white,rotate=36] (1,0) circle (.225) node {$2$};  
  \draw[fill=white,rotate=2*36] (1,0) circle (.225) node {$3$};
  \draw[fill=white,rotate=3*36] (1,0) circle (.225) node {$4$};
  \draw[fill=white,rotate=4*36] (1,0) circle (.225) node {$5$};
  \draw[fill=white,rotate=5*36] (1,0) circle (.225) node {$6$};
  \draw[fill=white,rotate=6*36] (1,0) circle (.225) node {$7$};
  \draw[fill=white,rotate=7*36] (1,0) circle (.225) node {$8$};
  \draw[fill=white,rotate=8*36] (1,0) circle (.225) node {$9$};
  \draw[fill=white,rotate=9*36] (1,0) circle (.225) node {$10$};
\end{tikzpicture}
\caption{Cycle graph $C_{10}$.}\label{fig:C10}
\end{figure}
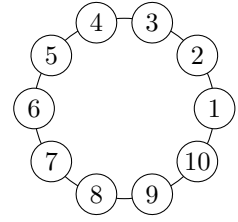
\begin{figure}
\centering
\input{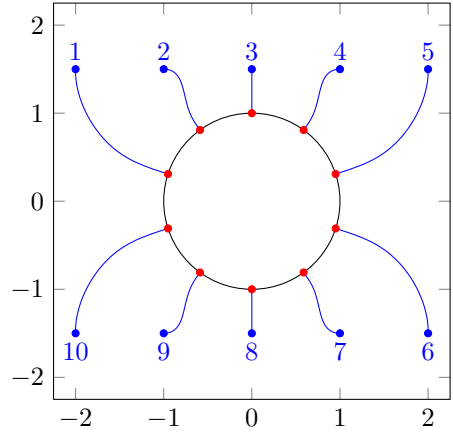}\hphantom{-1-1-!}
\caption{Numerical solution of \eqref{eq:circODE} for the cycle graph $C_{10}$.}\label{fig:Efig1}
\end{figure}
\end{example}

The circle can be continuously deformed into any Jordan curve, making the control from the foregoing example applicable to a wide range of formations in the plane. For $M$ being an ellipse, the way in which our control must be adapted is particularly simple, as we briefly describe in the next example.

\begin{example}\label{ex:ellipse}
 Let $M$ be an ellipse in $\mathbb{R}^{2}$ with radius $a$ in the first coordinate and radius $1$ in the second coordinate. As in the previous example, points from the ellipse can be injected onto the special orthogonal group $\operatorname{SO}\left(2\right)$ via
\begin{equation}
 \begin{bmatrix}a\operatorname{cos}\left(\alpha\right) \\ \hphantom{a}\operatorname{sin}\left(\alpha\right)\end{bmatrix}\mapsto\begin{bmatrix}\operatorname{cos}\left(\alpha\right) & -\operatorname{sin}\left(\alpha\right) \\ \operatorname{sin}\left(\alpha\right) & \hphantom{-}\operatorname{cos}\left(\alpha\right)\end{bmatrix}
\end{equation}
and thus all considerations regarding the circle remain correct. Using this representation, the only changes that are required in \eqref{eq:circODE} are that arguments of norms must be 
\begin{equation}
\begin{bmatrix}1/a & 0 \\ 0 & 1\end{bmatrix}x_{i}\;\;\;\text{ and }\;\;\;\begin{bmatrix}1/a & 0 \\ 0 & 1\end{bmatrix}x_{j}
\end{equation}
instead of $x_{i}$ and $x_{j}$, respectively. This has the effect of retracting points onto the ellipse instead of the circle. Also, the inverse of the skew-symmetric matrix that the logarithmic map returns must undergo the similarity transform of being multiplied by
\begin{equation}
\begin{bmatrix}1 & 0 \\ 0 & 1/a\end{bmatrix}\;\;\;\text{and}\;\;\;\begin{bmatrix}1 & 0 \\ 0 & a\end{bmatrix}
\end{equation}
from the left and right, respectively, before multiplying it with $x_{i}$. This has the purpose of making the resulting vector tangent to the ellipse instead of the circle. We now consider $n=12$ systems coupled through the unweighted cycle graph $C_{12}$
. With this choice of graph, we solved \eqref{eq:circODE}, modulo the above-mentioned substitutions, numerically for some initial condition; the numerical solutions are plotted in Fig. \ref{fig:Efig3}. As before, the initial condition is indicated by blue circles (\raisebox{1pt}{\tikz{\filldraw[blue] (0,0) circle (.06);}}) and the limiting point is marked with red circles (\raisebox{1pt}{\tikz{\filldraw[red] (0,0) circle (.06);}}). We find that the positions approach an evenly spaced configuration on the ellipse, as desired.
\begin{figure}
\centering
\input{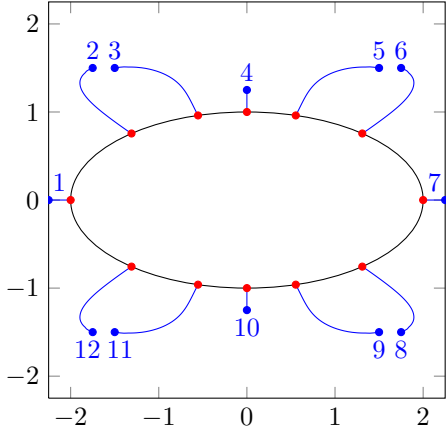}\hphantom{-1-1-!}
\caption{Numerical solution of \eqref{eq:circODE}, modulo the substitutions from Example \ref{ex:ellipse}, for the cycle graph $C_{12}$.}\label{fig:Efig3}
\end{figure}
\end{example}

Our methods proved to be successful for two exemplary formations in the plane $\mathbb{R}^{2}$. While the interest in planar formations largely stems from vehicle platoons or robot swarms, interest in spatial formations in $\mathbb{R}^{3}$ is readily justified, e.g. by formation flights. The simplest compactly and smoothly embedded submanifold of interest in $\mathbb{R}^{3}$ should be the ($2$-)sphere. This shall be reason enough to consider a spherical formation in $\mathbb{R}^{3}$ in the forthcoming example. We also recall that the classical Thomson and Fekete problems are posed as finding evenly distributed points on the sphere in $\mathbb{R}^{3}$. 

\begin{example}\label{ex:sphere}
Let $M$ be the unit sphere in $\mathbb{R}^{3}$. Most considerations from Example \ref{ex:circ} remain correct, though the notation does not remain as simple. More specifically, the representation
\begin{align}
&\begin{bmatrix}\operatorname{sin}\left(\alpha^{2}\right)\operatorname{cos}\left(\alpha^{1}\right) \\ \operatorname{sin}\left(\alpha^{2}\right)\operatorname{sin}\left(\alpha^{1}\right) \\ \operatorname{cos}\left(\alpha^{2}\right)\end{bmatrix}\mapsto \label{eq:sphereinj} \\  & \hphantom{XXX}\begin{bmatrix}\operatorname{cos}\left(\alpha^{2}\right)\operatorname{cos}\left(\alpha^{1}\right) & -\operatorname{sin}\left(\alpha^{1}\right) & \operatorname{sin}\left(\alpha^{2}\right)\operatorname{cos}\left(\alpha^{1}\right) \\ \operatorname{cos}\left(\alpha^{2}\right)\operatorname{sin}\left(\alpha^{1}\right) & \hphantom{-}\operatorname{cos}\left(\alpha^{1}\right) & \operatorname{sin}\left(\alpha^{2}\right)\operatorname{sin}\left(\alpha^{1}\right) \\ -\operatorname{sin}\left(\alpha^{2}\right) & 0 & \operatorname{cos}\left(\alpha^{2}\right)\end{bmatrix} \nonumber
\end{align}
of the sphere in $\operatorname{SO}\left(3\right)$ is an injection that does not attain every value in $\operatorname{SO}\left(3\right)$. The retraction \eqref{eq:circretr} remains the same and we denote the above representation \eqref{eq:sphereinj} of some retracted $x_{i}$ as a member of $\operatorname{SO}\left(3\right)$ by $R_{i}$. Using this notation, we may still apply the logarithmic map $\operatorname{log}:\operatorname{SO}\left(3\right)\to\mathfrak{so}\left(3\right)$ to $R_{j}^{\top}R_{i}$ in order to find the (initial) velocity vector of the geodesic joining $R_{i}$ and $R_{j}$ but now that the tangent space is not one-dimensional, that velocity vector may not be inverted. Instead, using the identity $2d\left(x_{i},x_{j}\right)^{2}=-\operatorname{tr}\hspace{-.3pt}\left(\vphantom{x_{j}^{2}}\!\hspace{-.3pt}\right.\operatorname{log}\left(\vphantom{x_{j}}\!\hspace{-1.3pt}\right.R_{j}^{\top}R_{i}\!\hspace{-1pt}\left.\vphantom{x_{i}}\right)^{2}\!\hspace{-.5pt}\left.\vphantom{x_{i}^{2}}\right)$, we obtain
\begin{align}
 \dot{x}_{i}&=\left(\frac{1-\left\|x_{i}\right\|}{\left\|x_{i}\right\|}\right)x_{i} \nonumber \\
 &\hspace{1pt}+\sum_{j=1}^{n}\frac{2W_{ij}}{\left\|x_{i}\right\|\operatorname{tr}\hspace{-.3pt}\left(\vphantom{x_{j}^{2}}\!\hspace{-.3pt}\right.\operatorname{log}\left(\vphantom{x_{j}}\!\hspace{-1.3pt}\right.R_{j}^{\top}R_{i}\!\hspace{-1pt}\left.\vphantom{x_{i}}\right)^{2}\!\hspace{-.5pt}\left.\vphantom{x_{i}^{2}}\right)}\operatorname{log}\left(R_{j}^{\top}R_{i}\right)x_{i}. \label{eq:sphereODE}
\end{align}
Now consider $n=5$ systems coupled through the unweighted complete graph $K_{5}$, i.e., $W_{ij}=W_{ji}=1$ for $j\neq i$; the graph is depicted in Fig. \ref{fig:K5}. With this choice of graph, we solved \eqref{eq:sphereODE} numerically for some initial condition; the numerical solutions are plotted in Fig. \ref{fig:S2fig}. The initial condition is indicated by blue circles (\raisebox{1pt}{\tikz{\filldraw[blue] (0,0) circle (.06);}}) and the limiting point is marked with red circles (\raisebox{1pt}{\tikz{\filldraw[red] (0,0) circle (.06);}}). We find that the positions approach the vertices of a triangular bipyramid, which are indeed known to be Fekete points.

\begin{figure}
\centering
\begin{tikzpicture}[scale=1.2]
\draw (0,0) circle (1);
\draw (1,0) -- (-1,.666);
\draw (1,0) -- (-1,-.666);
\draw (.333,1) -- (-1,-.666);
\draw (.333,1) -- (.333,-1);
\draw (.333,-1) -- (-1,.666);
\draw[fill=white,rotate=0*72] (1,0) circle (.225) node {$1$};
  \draw[fill=white,rotate=1*72] (1,0) circle (.225) node {$2$}; 
  \draw[fill=white,rotate=2*72] (1,0) circle (.225) node {$3$};
  \draw[fill=white,rotate=3*72] (1,0) circle (.225) node {$4$};
  \draw[fill=white,rotate=4*72] (1,0) circle (.225) node {$5$};
  \end{tikzpicture}
\caption{Complete graph $K_{5}$.}\label{fig:K5}
\end{figure}
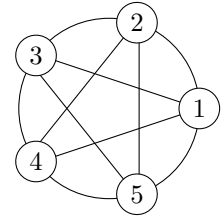
\begin{figure}
\centering
\input{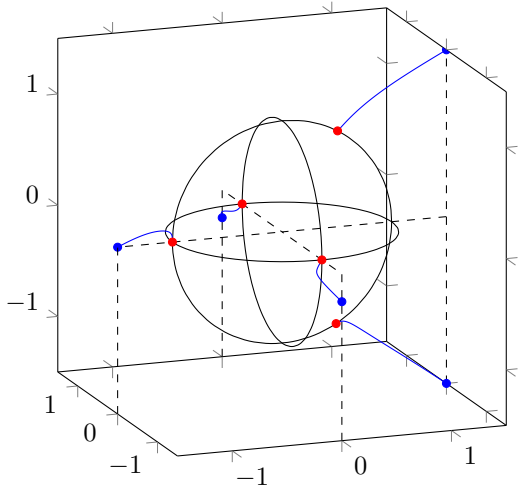}\hphantom{-1-1-!}
\caption{Numerical solution of \eqref{eq:sphereODE} for the complete graph $K_{5}$.}\label{fig:S2fig}
\end{figure}

\end{example}

\section{Graph Theoretical Interpretation \\ of Equilibria}\label{sec:graph}
In the maximization of $\phi$, the parameters $W_{ij}$, which can be interpreted as being determined by a weighted, undirected graph, play a crucial role. In the previous section, we saw that the cycle graph was well suited for evenly spacing points on the circle (Example \ref{ex:circ}), and that the complete graph brought positions to the Fekete points on the sphere (Example \ref{ex:sphere}). It shall be emphasized that the cycle graph does not bring the positions towards the Fekete points on the sphere and that we also encounter difficulties when employing the complete graph on the circle, as illustrated in our next example.
\begin{example}\label{ex:graphs}
Let $M$ be the unit circle in $\mathbb{R}^{2}$ and consider $n=6$ systems coupled through the complete graph $K_{6}$. Solving \eqref{eq:circODE} numerically, one finds that oscillations occur {that} grow stronger as the positions approach the circle. The numerical solutions are plotted in Fig. \ref{fig:Efig_complete} with initial condition indicated by blue circles (\raisebox{1pt}{\tikz{\filldraw[blue] (0,0) circle (.06);}}) and the configuration for some large time is marked with red circles (\raisebox{1pt}{\tikz{\filldraw[red] (0,0) circle (.06);}}). The oscillations are magnified for better visibility.
\begin{figure}
\centering
\includegraphics{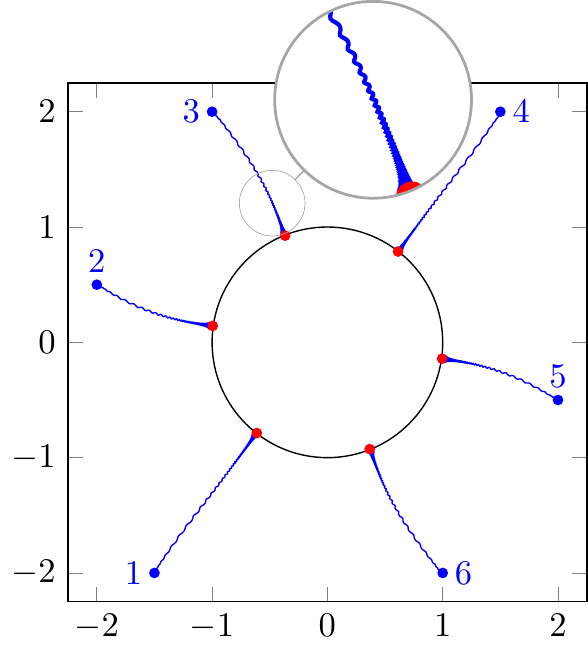}\hphantom{-1-1-!}
\caption{Numerical solution of \eqref{eq:circODE} for the complete graph $K_{6}$.}\label{fig:Efig_complete}
\end{figure}
This behavior is explained by verifying whether the evenly spaced configuration is an equilibrium of \eqref{eq:circODE}. In fact, introducing the notation 
\begin{equation}
\alpha_{ij}\Omega =\operatorname{log}\left(\begin{bmatrix}x_{i}\cdot x_{j} & x_{i}\cdot\Omega x_{j} \\ x_{j}\cdot\Omega x_{i} & x_{i}\cdot x_{j}\end{bmatrix}\right) \label{eq:angledef}
\end{equation}
for the directed angle between two points $x_{i}$, $x_{j}$ on the circle {that} are neighbors in the graph under consideration (i.e., for which $W_{ij}\neq 0$), we figure that
\begin{equation*}
 \alpha_{12}=-\frac{2\pi}{6},\hspace{1pt}\alpha_{13}=-\frac{2\pi}{3},\hspace{1pt}\alpha_{14}=\pm\pi,\hspace{1pt}\alpha_{15}=\frac{2\pi}{3},\hspace{1pt}\alpha_{16}=\frac{2\pi}{6}
\end{equation*}
should asymptotically hold from Fig. \ref{fig:Efig_complete}. If one now asks whether this configuration is indeed an equilibrium of \eqref{eq:circODE}, then one finds that
\begin{equation}
\frac{1}{\alpha_{12}}+\frac{1}{\alpha_{13}}+\frac{1}{\alpha_{14}}+\frac{1}{\alpha_{15}}+\frac{1}{\alpha_{16}}\neq 0
\end{equation}
whence the answer is negative. At the same time, one finds that removal of the edge between the vertices $1$ and $4$ would indeed turn this point into an equilibrium of $\dot{x}_{1}$, and similarly we would have to remove the edges $\left(2,5\right)$ and $\left(3,6\right)$ in order to establish an equilibrium for all positions. Doing so, we arrive at a $4$-regular graph (with $6$ vertices) and solving \eqref{eq:circODE} again for this graph, we find that the oscillations observed before no longer occur. One is thus tempted to think that \emph{regular} graphs are suited best for our evenly spaced configurations on the circle, particular when recalling that circulant graphs play a crucial role in \cite{Sepulchre2008} for stabilization of circular formations. However, consider the Thomsen (``utility'') graph depicted in Fig. \ref{fig:util}, which is $3$-regular (and also complete bipartite). Solving \eqref{eq:circODE} numerically for this graph and plotting the numerical solutions in Fig.\ {\ref{fig:Efig_UTILITY}} (initial condition again indicated by blue circles (\raisebox{1pt}{\tikz{\filldraw[blue] (0,0) circle (.06);}}) and configuration for some large time marked with red circles (\raisebox{1pt}{\tikz{\filldraw[red] (0,0) circle (.06);}})), we find that the positions do not come to rest. Instead, the positions enter a periodic orbit on the circle whilst being evenly spaced thereon. Let us try to explain this as we did above. To this end, first notice that
\begin{equation}
 \alpha_{14}=\pm\pi,\;\;\;\alpha_{15}=\frac{2\pi}{6},\;\;\;\alpha_{16}=-\frac{2\pi}{6}
\end{equation}
should asymptotically hold (again inferred from the plot). But as the reciprocals thereof do not sum up to zero, this configuration, again, does not constitute an equilibrium of \eqref{eq:circODE}. We would have to delete the edges $\left(1,4\right)$, $\left(2,5\right)$, and $\left(3,6\right)$ to let this happen. If we removed these edges, we again arrived at the cycle graph $C_{6}$, for which the points indeed come to rest at an evenly distributed configuration (as we saw for $n=10$ in Example \ref{ex:circ}). In conclusion, we find that $k$-regular graphs, with $k$ an even positive number, are suited well for evenly spaced circular formations. These graphs are precisely the regular graphs possessing Eulerian cycles.
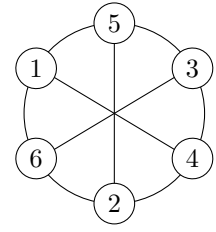
\begin{figure}
\centering
\begin{tikzpicture}[scale=1.2]
\draw (0,0) circle (1);
\draw (0,1) -- (0,-1);
\draw (-1,.6) -- (1,-.6);
\draw (-1,-.6) -- (1,.6);
\draw[fill=white,rotate=90+60] (1,0) circle (.225) node {$1$};
  \draw[fill=white,rotate=90] (1,0) circle (.225) node {$5$}; 
  \draw[fill=white,rotate=30] (1,0) circle (.225) node {$3$};
  \draw[fill=white,rotate=-30] (1,0) circle (.225) node {$4$};
  \draw[fill=white,rotate=-30-60] (1,0) circle (.225) node {$2$};
  \draw[fill=white,rotate=-30-120] (1,0) circle (.225) node {$6$};
 
  \end{tikzpicture}
\caption{Thomsen (``utility'') graph.}\label{fig:util}
\end{figure}
\begin{figure}
\centering
\input{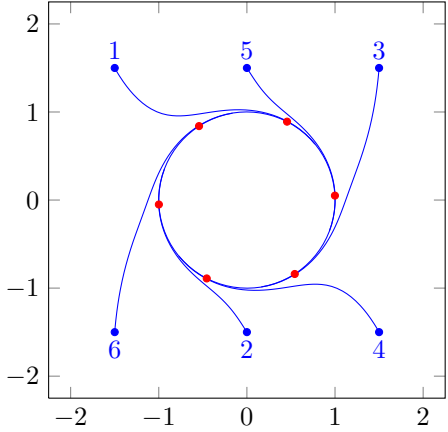}\hphantom{-1-1-!}
\caption{Numerical solution of \eqref{eq:circODE} for the Thomsen graph.}\label{fig:Efig_UTILITY}
\end{figure}
\end{example}

The previous example provided some insight into the the role of graph theory for equilibria of \eqref{eq:circODE}. Next, we generalize these observations. To this end, we adopt the notation \eqref{eq:angledef}. Equating the right hand side of \eqref{eq:circODE} with zero, we arrive at all $x_{i}$ being on the circle and the angles $\alpha_{ij}$ satisfying 
\begin{equation*}
{\fontsize{9pt}{1em}\selectfont\begin{bmatrix}0 & W_{12}/\alpha_{12} & \cdots & \cdots & W_{1n}/\alpha_{1n} \\ W_{21}/\alpha_{21} & 0 & W_{23}/\alpha_{23} & \cdots & W_{2n}/\alpha_{2n} \\ \vdots & W_{32}/\alpha_{32} & 0 & & \vdots \\ \vdots & \vdots & & \ddots &  \\ W_{n1}/\alpha_{n1} & W_{n2}/\alpha_{n2} & \cdots &  & 0\end{bmatrix}\!\!\begin{bmatrix}\\[-.9em] 1 \\[-.5em] \vdots \\[-.25em] \vdots \\[-.25em] \vdots \\ 1\end{bmatrix}}\!=0.
 \end{equation*}
But we defined the function $\left(i,j\right)\mapsto W_{ij}$ to be symmetric (i.e., our graph to be undirected). Knowing that the logarithm of the transpose of a rotation matrix is just the negative logarithm of that rotation matrix, we further find that $\alpha_{ij}=-\alpha_{ji}$. Substituting these findings into our last equation, we find that  
\begin{equation*}
{\fontsize{9pt}{1em}\selectfont\begin{bmatrix}\hphantom{-}0\! & \hphantom{-}W_{12}/\alpha_{12}\! & \cdots\! & \cdots\! & W_{1n}/\alpha_{1n} \\ -W_{12}/\alpha_{12}\! & 0\! & W_{23}/\alpha_{23}\! & \cdots\! & W_{2n}/\alpha_{2n} \\ \hphantom{-}\vdots\! & -W_{23}/\alpha_{23}\! & 0\! & \! & \vdots \\ \hphantom{-}\vdots\! & \vdots\! & \! & \ddots\! &  \\ -W_{1n}/\alpha_{1n}\! & -W_{2n}/\alpha_{2n}\! & \cdots\! & \! & 0\end{bmatrix}\!\!\begin{bmatrix}\\[-.9em] 1 \\[-.5em] \vdots \\[-.25em] \vdots \\[-.25em] \vdots \\ 1\end{bmatrix}}\!=0
\end{equation*}
must hold for a configuration to make an equilibrium. Next, noticing that the matrix on the left is skew-symmetric, we know that it can be written as a linear combination of the generators $\Omega_{ij}$ of the Lie algebra $\mathfrak{so}\left(n\right)$, where we employ the convention that the $j$th entry of the $i$th row of $\Omega_{ij}$ is $1${, i.e., $\Omega_{ij}=e_{i}e_{j}^{\top}-e_{j}e_{i}^{\top}$}. Thus, we have that
\begin{equation}
 \sum_{j>i}\frac{W_{ij}}{\alpha_{ij}}\Omega_{ij}\begin{bmatrix}1 \\ \vdots \\ 1\end{bmatrix}= \sum_{j>i}\frac{W_{ij}}{\alpha_{ij}}\left(e_{i}-e_{j}\right)=0 \label{eq:skewsum}
\end{equation}
must hold at an equilibrium, where $e_{i}\in\mathbb{R}^{n}$ denotes the $i$th vector of the standard basis. Letting $E$ denote the weighted incidence matrix of our graph, i.e., the matrix whose columns are the nonzero vectors $W_{ij}\left(e_{i}-e_{j}\right)$, with lexicographically ordered indices $\left(i,j\right)$, $j>i$, we are now ready to state the following proposition.
\begin{proposition}\label{prop:incidence}
Let all $x_{i}$ lie on the circle. Denote by $\alpha$ the vector whose entries are $1/\alpha_{ij}$, with $\alpha_{ij}$ defined as in \eqref{eq:angledef}, and lexicographically ordered indices $\left(i,j\right)$, $j>i$. Denote the weighted incidence matrix of the undirected, weighted graph associated to the symmetric function $\left(i,j\right)\mapsto W_{ij}$ by $E$. Then $x$ is an equilibrium of \eqref{eq:circODE} if and only if $\alpha$ is in the nullspace of $E$.
\end{proposition}
\begin{proof}
The condition $E\alpha=0$ is equivalent to \eqref{eq:skewsum}.
\end{proof}
Let us assume for a moment that our graph is unweighted, i.e., that all nonzero $W_{ij}$ are equal to $1$. Then, the incidence matrix $E$ can be seen as a matrix over the Galois field $\operatorname{GF}\left(3\right)$. Taking this point of view, a cycle in our graph is a collection of columns of $E$ {that} are linearly dependent over $\operatorname{GF}\left(3\right)$, i.e., each cycle can be though of as a vector $c$ over $\operatorname{GF}\left(3\right)$ for which $Ec=0$. These cycles (vectors) $c$ constitute a vector space over $\operatorname{GF}\left(3\right)$, the nullspace of $E$ (over $\operatorname{GF}\left(3\right)$), which is called the \emph{cycle space} of the graph \cite{Godsil2001}. Its dimension is, in this sense, the number of linearly independent (over $\operatorname{GF}\left(3\right)$) cycles in the graph. Should we restrict our attention to unweighted graphs, then the foregoing proposition tells us that an equilibrium configuration must consist of reciprocal angles lying in the cycle space (over $\mathbb{R}$) of our graph and as the dimension of our cycle space increases, the possible number of 
equilibrium configurations increases, as well. On the other hand, graphs with trivial cycle spaces, such as line graphs (or any acyclic graphs), do not admit equilibria of \eqref{eq:circODE} whatsoever as $\alpha$ can not be zero (in fact, $\alpha$ does not only have to lie in the nullspace of $E$, but also in the cone of vectors {without} zero entries). In this context, recall that we had observed a connection between evenly spaced equilibrium configurations and regular graphs with Eulerian cycles in Example \ref{ex:graphs}. Now, having encountered the above algebraic characterization of equilibria, we are ready to generalize and formalize this observation. In fact, the regularity assumption may be omitted.

\begin{corollary}\label{cor:Eulerian}
Let all $x_{i}$ lie on the circle. Define $\alpha_{ij}$ as in \eqref{eq:angledef}. If the undirected, unweighted graph associated to the symmetric function $\left(i,j\right)\mapsto W_{ij}$ possesses an Eulerian cycle (equivalently, if every vertex has even and positive degree), then there is an equilibrium $x$ of \eqref{eq:circODE} such that all $\alpha_{ij}$ have the same absolute value.
\end{corollary}
\begin{proof}
An Eulerian cycle is a vector $c$ over $\operatorname{GF}\left(3\right)$ from the nullspace of the incidence matrix $E$ with the property that all of its entries are either $1$ or $-1$. Recalling Proposition \ref{prop:incidence}, the claim remains proven. 
\end{proof}

We have seen that, when all positions $x_{i}$ are on the circle, then $E\alpha =0$ is a necessary and sufficient condition for $x$ to be an equilibrium of \eqref{eq:circODE}. Our next example explains why only solving $E\alpha=0$ alone (without having \eqref{eq:angledef} in mind) is necessary, but not sufficient.
\begin{example}
Consider the cycle graph $C_{n}$. Following our above convention of lexicographically ordering the vectors $e_{i}-e_{j}$ according to $\left(i,j\right)$, $j>i$, its incidence matrix $E$ has the columns $e_{1}-e_{2}$, $e_{1}-e_{n}$, $e_{2}-e_{3}$, $e_{3}-e_{4}$, and so forth. The equation $E\alpha=0$ thus reads
\begin{equation}
 \begin{bmatrix}\hphantom{-}1 & \hphantom{-}1 & \hphantom{-}0 & \hphantom{-}0 & \\ -1 & \hphantom{-}0 & \hphantom{-}1 & \hphantom{-}0 &  \\ \hphantom{-}0 & \hphantom{-}0 & -1 & \hphantom{-}1  & \\ \hphantom{-}0 & \hphantom{-}0 & \hphantom{-}0 & -1 & \\ \hphantom{-}0 & \hphantom{-}0 & \hphantom{-}0 & \hphantom{-}0  & \cdots\, \\ \hphantom{-}\vdots & \hphantom{-}\vdots & \hphantom{-}\vdots & \hphantom{-}\vdots  &  \\ \hphantom{-}0 & \hphantom{-}0 & \hphantom{-}0 & \hphantom{-}0  &  \\ \hphantom{-}0 & -1 & \hphantom{-}0 & \hphantom{-}0 & \end{bmatrix}\begin{bmatrix}1/\alpha_{12} \\ 1/\alpha_{1n} \\ 1/\alpha_{23} \\ 1/\alpha_{34} \\ \vdots \end{bmatrix}=0 \label{eq:cyclereciprocang}
\end{equation}
from which we infer that $\alpha$ must be in the ($1$-dimensional) cycle space of $C_{n}$, spanned by the vector $\left(1,-1,1,\dots,1\right)$. This implies that our solution is of the form
\begin{equation}
 \alpha_{12}=-\alpha_{1n}=\alpha_{23}=\alpha_{34}=\cdots
\end{equation}
and hence uniquely determined by, say, $\alpha_{12}$. But not every solution of this form can be realized as points $x_{i}$ on the circle such that \eqref{eq:angledef} is satisfied. 
Recalling that the quantities $\alpha_{ij}$ in fact represent angles, we arrive at the additional ``physical'' constraint requiring that
\begin{equation}
B^{\top}\begin{bmatrix}\alpha_{12} \\ \alpha_{1n} \\ \alpha_{23} \\ \alpha_{34} \\ \vdots \end{bmatrix}\operatorname{mod}2\pi=0\label{eq:mod2pi}
\end{equation}
holds, where $B$ is a matrix over $\operatorname{GF}\left(3\right)$ whose columns are a basis for the cycle space of $C_{n}$, i.e., $EB=0$. The condition can be seen as the algebraic equivalent of insisting that the angles $\alpha_{ij}$ correspond to some realization of the positions $x_{i}$ such that they lie on the circle. This being said, together with the necessary condition $E\alpha=0$, the latter equation constitutes a sufficient condition for equilibria of \eqref{eq:circODE}. In the present example, $B$ is just the vector $\left(1,-1,1,\dots,1\right)$ and we conclude that $\left(n\alpha_{12}\right)\operatorname{mod}2\pi=0$ which is indeed equivalent to the solution $\alpha_{12}=2\pi/n$ observed in Example \ref{ex:circ}. 
\end{example}

Our previous example revealed that the condition $E\alpha=0$ together with \eqref{eq:mod2pi} is sufficient to characterize equilibria of \eqref{eq:circODE}. Still, both equations are hard to solve explicitly. As for $E\alpha =0$, this complication stems from the fact that solutions must be contained in the cone of vectors {that} have all entries nonzero, caused by the reciprocal angles appearing in $\alpha$. Yet, this complication is (partially) overcome by multiplying the $p$th element of $E\alpha$ with
\begin{equation}
\prod_{\substack{W_{ip}\neq 0\\p>i}}\alpha_{ip}\prod_{\substack{W_{pj}\neq 0\\j>p}}\alpha_{pj},  \label{eq:allangles}
\end{equation}
turning it into a symmetric polynomial, viz. symmetric with respect to permutations of the indices $\lbrace i\neq p\hspace{1pt}|W_{ip}\neq 0\rbrace$. The zeros of the resulting polynomials, i.e., the solutions to $DE\alpha=0$, with $D\in\mathbb{R}^{n\times n}$ being the (full-rank) diagonal matrix {that} has the entries \eqref{eq:allangles}, $p=1,\dots,n$ on its diagonal, hence characterize our equilibria, as well, which is summarized in the following proposition.
\begin{proposition}\label{prop:poly}
 Let all $x_{i}$ lie on the circle and let $\alpha_{ij}$ be defined as in \eqref{eq:angledef}. Then $x$ is an equilibrium of \eqref{eq:circODE} if and only if all $\alpha_{ij}$ are zeros of the $n$ polynomials resulting from multiplying $\sum_{i>p}W_{pi}/\alpha_{pi}-\sum_{j<p}W_{jp}/\alpha_{jp}$, with \eqref{eq:allangles}, $p=1,\dots,n$. The zeros of the $p$th polynomial are invariant under permutations of $\lbrace i\neq p\hspace{1pt}|W_{ip}\neq 0\rbrace$.
\end{proposition}
\begin{proof}
Recall Proposition \ref{prop:incidence}. We notice that $\sum_{i>p}W_{pi}/\alpha_{pi}-\sum_{j<p}W_{jp}/\alpha_{jp}$ is the $p$th element of $E\alpha$. As our angles $\alpha_{ij}$ cannot be zero, multiplication of $E\alpha=0$ with the diagonal matrix $D$ having \eqref{eq:allangles} as its $p$th diagonal element does not change the zeros of the resulting system of equations. Symmetry with respect to permutations of $\lbrace i\neq p\hspace{1pt}|W_{ip}\neq 0\rbrace$ follows from finding that \eqref{eq:allangles} remains invariant under such permutations.
\end{proof}

Our next example illustrates how the procedure suggested in the foregoing proposition can turn out useful for computation of equilibria of \eqref{eq:circODE}, particularly for graphs with large cycle spaces.

\begin{example}
Following the approach from our previous example, let us now try to characterize equilibria of \eqref{eq:circODE} for more complicated graphs: we consider $n=7$ and the Moser spindle depicted in Fig. \ref{fig:Moser}. Its incidence matrix is
\setcounter{MaxMatrixCols}{11}
\begin{equation*}
 E=
 \begin{bmatrix}
 \hphantom{-}1\!\!	&\hphantom{-}1\!\!	&\hphantom{-}1\!\!	&	\hphantom{-}0\!\!&	\hphantom{-}0\!\!&	\hphantom{-}0\!\!&	\hphantom{-}0\!\!&	\hphantom{-}0\!\!&	\hphantom{-}0\!\!&	\hphantom{-}0\!\!&  \hphantom{-}0\\
 -1\!\!	&\hphantom{-}0\!\!	&\hphantom{-}0\!\!	&\hphantom{-}1\!\!	&\hphantom{-}1\!\!	&	\hphantom{-}0\!\!&	\hphantom{-}0\!\!&	\hphantom{-}0\!\!&	\hphantom{-}0\!\!&	\hphantom{-}0\!\!& \hphantom{-}0 \\
\hphantom{-}0\!\!&	-1\!\!	&\hphantom{-}0\!\!	&-1\!\!	&\hphantom{-}0\!\!	&\hphantom{-}1\!\!	&	\hphantom{-}0\!\!&	\hphantom{-}0\!\!&	\hphantom{-}0\!\!&	\hphantom{-}0\!\!& \hphantom{-}0 \\
\hphantom{-}0\!\!&		\hphantom{-}0\!\!&	\hphantom{-}0\!\!&	\hphantom{-}0\!\!&-1\!\!	&-1\!\!	&\hphantom{-}1\!\!	&\hphantom{-}1\!\!	&	\hphantom{-}0\!\!&	\hphantom{-}0\!\!& \hphantom{-}0 \\
	\hphantom{-}0\!\!&	\hphantom{-}0\!\!&	\hphantom{-}0\!\!&	\hphantom{-}0\!\!&	\hphantom{-}0\!\!&	\hphantom{-}0\!\!&-1\!\!	&\hphantom{-}0\!\!	&\hphantom{-}1\!\!	&\hphantom{-}1\!\!	& \hphantom{-}0\\
	\hphantom{-}0\!\!&	\hphantom{-}0\!\!&	\hphantom{-}0\!\!&	\hphantom{-}0\!\!&	\hphantom{-}0\!\!&	\hphantom{-}0\!\!&	\hphantom{-}0\!\!&-1\!\!	&-1\!\!	&\hphantom{-}0\!\!	&\hphantom{-}1 \\
	\hphantom{-}0\!\!&	\hphantom{-}0\!\!&-1\!\!	&	\hphantom{-}0\!\!&	\hphantom{-}0\!\!&	\hphantom{-}0\!\!&	\hphantom{-}0\!\!&	\hphantom{-}0\!\!&	\hphantom{-}0\!\!&-1\!\!	&-1
 \end{bmatrix}
\end{equation*}
and its ($5$-dimensional) cycle space is the image of the matrix
\begin{equation}
 B=\begin{bmatrix}
\hphantom{-}1	&-1		&\hphantom{-}0	&\hphantom{-}1	&\hphantom{-}1 \\
-1		&\hphantom{-}1	&\hphantom{-}0	&\hphantom{-}0	&\hphantom{-}0 \\
\hphantom{-}0	&\hphantom{-}0	&\hphantom{-}0	&-1		&-1 \\
\hphantom{-}1	&\hphantom{-}0	&\hphantom{-}0	&\hphantom{-}0	&\hphantom{-}0 \\
\hphantom{-}0	&-1		&\hphantom{-}0	&\hphantom{-}1	&\hphantom{-}1 \\
\hphantom{-}0	&\hphantom{-}1	&\hphantom{-}0	&\hphantom{-}0	&\hphantom{-}0 \\
\hphantom{-}0	&\hphantom{-}0	&\hphantom{-}1	&\hphantom{-}1	&\hphantom{-}0 \\
\hphantom{-}0	&\hphantom{-}0	&-1		&\hphantom{-}0	&\hphantom{-}1 \\
\hphantom{-}0	&\hphantom{-}0	&\hphantom{-}1	&\hphantom{-}0	&\hphantom{-}0 \\
\hphantom{-}0	&\hphantom{-}0	&\hphantom{-}0	&\hphantom{-}1	&\hphantom{-}0 \\
\hphantom{-}0	&\hphantom{-}0	&\hphantom{-}0	&\hphantom{-}0	&\hphantom{-}1
 \end{bmatrix}
\end{equation}
over $\operatorname{GF}\left(3\right)$. Solving $E\alpha=0$ and \eqref{eq:mod2pi} at the same time, even numerically, turns out to be a hard task. One complication is that no entry of $\alpha$ can be zero in the seemingly linear equation $E\alpha=0$. This could be resolved by maximizing the support of $\alpha$ subject to $E\alpha=0$. Here, we opt to multiply the $p$th element of $E\alpha$ by \eqref{eq:allangles}, as suggested in Proposition \ref{prop:poly}, to obtain the symmetric polynomials
\begin{align*}
 \alpha_{13}\alpha_{17}+\alpha_{12}\alpha_{17}+\alpha_{12}\alpha_{13}=0, \\
 -\alpha_{23}\alpha_{24}+\alpha_{12}\alpha_{24}+\alpha_{12}\alpha_{23}=0, \\
  -\alpha_{23}\alpha_{34}-\alpha_{13}\alpha_{34}+\alpha_{13}\alpha_{23}=0, \\
    -\alpha_{34}\alpha_{45}\alpha_{46}-\alpha_{24}\alpha_{45}\alpha_{46}+\alpha_{24}\alpha_{34}\alpha_{46}+\alpha_{24}\alpha_{34}\alpha_{45}=0, \\
      -\alpha_{56}\alpha_{57}+\alpha_{45}\alpha_{57}+\alpha_{45}\alpha_{56}=0, \\ 
      -\alpha_{56}\alpha_{67}-\alpha_{46}\alpha_{67}+\alpha_{46}\alpha_{56}=0, \\
            -\alpha_{57}\alpha_{67}-\alpha_{17}\alpha_{67}-\alpha_{17}\alpha_{57}=0, 
\end{align*}
the $p$th polynomial only being symmetric with respect to permutations not involving $p$. Having this formulation at hand, we know that the angles for which $x$ is at rest must be contained in the intersection of the algebraic varieties containing the zeros of these polynomials, i.e., we may now solve these equations iteratively (by repeatedly intersecting these algebraic varieties), here obtaining
\begin{align*}
 \alpha_{12}=\alpha_{34}=\alpha_{45}=\alpha_{67}&=-2\left(5\right.+\sqrt{5}\!\!\hspace{-3pt}\left.\hphantom{3}\right)\pi/11\left(3\right.+\sqrt{5}\!\!\hspace{-3pt}\left.\hphantom{3}\right) ,\\ \alpha_{13}=\alpha_{24}=\alpha_{46}=\alpha_{57}&=\left(3\right.+\sqrt{5}\!\!\hspace{-3pt}\left.\hphantom{3}\right)\alpha_{12}/2,\\ \alpha_{17}&=2\left(\pi+\alpha_{12}+\alpha_{13}\right),\\ \alpha_{23}=\alpha_{56}&=\alpha_{13}-\alpha_{12},
\end{align*}
which is indeed a solution to $E\alpha=0$ as well as to \eqref{eq:mod2pi}. Indeed, solving \eqref{eq:circODE} numerically for some initial condition with this choice of graph and plotting the numerical solutions in Fig. \ref{fig:Efig_MOSER} (with initial condition again indicated by blue circles (\raisebox{1pt}{\tikz{\filldraw[blue] (0,0) circle (.06);}}) and limiting point again marked with red circles (\raisebox{1pt}{\tikz{\filldraw[red] (0,0) circle (.06);}})), this computation is confirmed as the positions approach precisely the configuration described by the above angles.

\begin{figure}
 \centering
 \begin{tikzpicture}[scale=1.2]
    \draw (0,0) circle (1);
    \draw (.66,-1) -- (0,1);
        \draw (-.66,-1) -- (0,1);
    \draw (-1,.33) -- (-.33,0);
    \draw (1,.33) -- (.33,0);
\draw[fill=white] (.33,0) circle (.225) node {$5$};
\draw[fill=white] (-.33,0) circle (.225) node {$3$};
\draw[fill=white,rotate=-90+36] (1,0) circle (.225) node {$7$};
  \draw[fill=white,rotate=72-90+36] (1,0) circle (.225) node {$6$}; 
  \draw[fill=white,rotate=2*72-90+36] (1,0) circle (.225) node {$4$};
  \draw[fill=white,rotate=3*72-90+36] (1,0) circle (.225) node {$2$};
  \draw[fill=white,rotate=4*72-90+36] (1,0) circle (.225) node {$1$};
\end{tikzpicture}
 \caption{The Moser spindle.}\label{fig:Moser}
\end{figure}
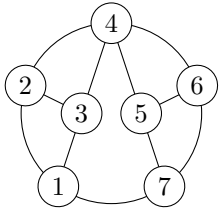

\begin{figure}
\centering
\input{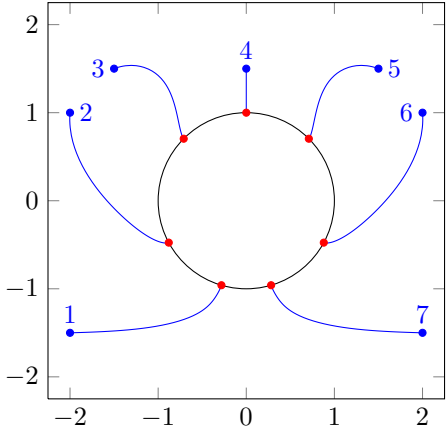}
\hphantom{-1-1-!}
\caption{Numerical solution of \eqref{eq:circODE} for the Moser spindle.}\label{fig:Efig_MOSER}
\end{figure}

\end{example}

An option which we did not exploit yet is to influence a formation by scaling the values of our nonzero weights $W_{ij}$. As our characterization $E\alpha=0$ reveals a rather explicit connection of these weights and equilibria of \eqref{eq:circODE}, namely that $\alpha_{ij}$ becomes $\left(W_{ij}^{\prime}/W_{ij}\right)\alpha_{ij}$ as we change $W_{ij}$ to $W_{ij}^{\prime}$, influencing individual angles by adapting the associated weights should be comparatively simple. In the following example, we exploit this observation to adjust the shape of a formation ad libitum.

\begin{example}
We exploit the possibility of adapting the weights $W_{ij}$ so as to eventually attain a desired configuration. Although we assumed that the weights may only be $0$ or $1$ for most of this paper, we now turn our attention to the case where we have nonidentical weights and discuss how solutions to $E\alpha=0$ are affected. Let us consider $8$ edges and suppose that our goal was to let the positions of systems $\left(1,2\right)$, $\left(3,4\right)$, $\left(5,6\right)$, and $\left(7,8\right)$ be pairwise close to each other but to still have these pairs be evenly spaced on the circle. Recalling equation \eqref{eq:cyclereciprocang}, it becomes evident that we must scale columns $1$, $4$, $6$, and $8$ $E$ with weights $W_{12}=W_{34}=W_{56}=W_{78}\ll 1$ in order to achieve this goal. Let us choose these weights to be $1/4$. Solving \eqref{eq:circODE} numerically for this choice of graph and plotting the numerical solutions in Fig. \ref{fig:Efig2} (with initial condition again indicated by blue circles (\raisebox{1pt}{\tikz{\filldraw[blue] (0,0) circle (.06);}}) and limiting point again marked with red circles (\raisebox{1pt}{\tikz{\filldraw[red] (0,0) circle (.06);}})), we find that the positions of systems $\left(1,2\right)$, $\left(3,4\right)$, $\left(5,6\right)$, and $\left(7,8\right)$ indeed move pairwise close to each other, but with the pairs being evenly spaced, as desired. In fact, evaluating $E\alpha=0$, we find that $\alpha_{23}=\alpha_{45}=\alpha_{67}=\alpha_{81}=4\alpha_{12}=4\alpha_{34}=4\alpha_{56}=4\alpha_{78}$. The condition \eqref{eq:mod2pi} remains satisfied for $B$ being a basis for the cycle space of the unweighted graph.

\begin{figure}
\centering
\input{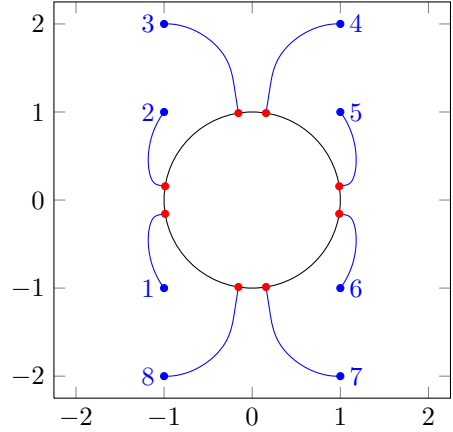}
\hphantom{-1-1-!}
\caption{Numerical solution of \eqref{eq:circODE} for the cycle graph $C_{8}$ but with $W_{12}=W_{34}=W_{56}=W_{78}=1/4$.}\label{fig:Efig2}
\end{figure}

\end{example}

In the light of Corollary \ref{cor:Eulerian}, the observation from the foregoing example can be expressed in a more general fashion.

\begin{corollary}
Let all $x_{i}$ lie on the circle. Define $\alpha_{ij}$ as in \eqref{eq:angledef}. If the undirected, weighted graph associated to the symmetric function $\left(i,j\right)\mapsto W_{ij}$ possesses an Eulerian cycle (equivalently, if every vertex has even and positive degree), then there is an equilibrium $x$ of \eqref{eq:circODE} such that all $W_{ij}/\alpha_{ij}$ have the same absolute value.
\end{corollary}
\begin{proof}
The claim is proven alike Corollary \ref{cor:Eulerian}: let $c$ be an Eulerian cycle, i.e., every entry of $c$ is either $1$ or $-1$ and $EW^{-1}c=0$ for $E$ being the weighted incidence matrix and $W$ being the diagonal matrix that has the weights $W_{ij}$, with lexicographically ordered indices $\left(i,j\right)$, $j>i$, as its diagonal entries (since $EW^{-1}$ is the unweighted incidence matrix). Now Proposition \ref{prop:incidence} tells us that $E\alpha=0$ classifies all equilibria. Thus $\alpha=W^{-1}c$, and hence $W\alpha=c$, defines an equlibrium and thus the claim is proven. 
\end{proof}

Until now, we restricted our attention to $M$ being the circle but we initially said that any formula which applies to the circle can be continuously transformed into a formula applicable to a (smooth) Jordan curve. Thus, we now briefly consider the case where $M$ is the image of some Jordan curve $\gamma:\left[0,1\right]\to\mathbb{R}^{2}$, similar to the efforts taken in \cite{Zhang2007}. Instead of the definition for the angles \eqref{eq:angledef}, we must now introduce the scalars $S_{ij}$ such that they satisfy
\begin{align}
 \left(S_{j}-S_{ij}\right)\operatorname{mod}1&=S_{i}\hspace{1pt}\;\;\;\text{ if }\;\;\;S_{i}-S_{j}\leq1/2,\\
 \left(S_{i}+S_{ij}\right)\operatorname{mod}1&=S_{j}\;\;\;\text{ if }\;\;\;S_{j}-S_{i}\leq1/2,
\end{align}
where the scalars $S_{i}$ are defined as
\begin{equation}
S_{i}:=\left.\gamma\right|_{\left[0,1\right)}^{-1}\!\left(r\left(x_{i}\right)\right),
\end{equation}
and replace our previous formulations $E\alpha=0$ and \eqref{eq:mod2pi} by the conditions 
\begin{equation}
E\begin{bmatrix}1/S_{12} \\ 1/S_{13} \\ 1/S_{14} \\ \vdots \end{bmatrix}=0\;\;\;\text{ and }\;\;\;B^{\top}\begin{bmatrix}S_{12} \\ S_{13} \\ S_{14} \\ \vdots \end{bmatrix}\operatorname{mod}1=0,
\end{equation}
respectively, where the entries $1/S_{ij}$ and $S_{ij}$ (only occurring if the corresponding weight $W_{ij}$ is nonzero) are lexicographically ordered according to the indices $\left(i,j\right)$, $j>i$.

\section{Balancing on the Special Euclidean Group}\label{sec:Eucl}
In formation control, one sometimes wishes to associate an attitude to a system in addition to its position. For instance, in formation flight, one would not only want that the positions of the airplanes arrange in a certain shape, but also that their heading angles agree. This interest is reflected by recent efforts to extend the formation control algorithms based upon rigid frameworks to the special Euclidean group \cite{Zelazo,Zhao}. In this section, we thus enhance the technique proposed in section \ref{sec:control} with the capability of taking orientations into account. That is, we consider formation control problems in the special Euclidean group $\operatorname{SE}\left(m\right)$ (one would expect that $m$ then is either $2$ or $3$). To this end, let the desired shape which describes our formation be an compactly and smoothly embedded submanifold $M$ of $\operatorname{SE}\left(m\right)$. Our formation control problem can then be cast as asymptotically bringing the poses (positions and attitudes) of our systems to an evenly spaced configuration on $M$ in a stable fashion. If one replaces the special Euclidean group with the sphere, then this idea is conceptually related to the approach taken in \cite{Paley2009}. 

In order to adapt our approach from section \ref{sec:control} to systems living on the special Euclidean group, we must first refresh our terminology. Let $T_{x}\operatorname{SE}\left(m\right)^{n}$ denote the tangent space of $\operatorname{SE}\left(m\right)^{n}$ at $x$. Then $T_{x}M^{n}$ is the subspace of $T_{x}\operatorname{SE}\left(m\right)^{n}$ consisting of velocity vectors tangent to $M^{n}$ at $x$ and  $N_{x}M^{n}$ is the orthogonal complement of $T_{x}M^{n}$ in $T_{x}\operatorname{SE}\left(m\right)^{n}$. The normal bundle of $M^{n}$ is a the vector bundle composed of the fibers $N_{x}M^{n}$, $x\in M^{n}$. Now, for some $x\in\operatorname{SE}\left(m\right)^{n}$, employ the notation $x=\left(R,p\right)$ with $R\in\operatorname{SO}\left(m\right)^{n}$ and $p\in\mathbb{R}^{mn}$. Similarly, denote a vector from $T_{x}\operatorname{SE}\left(m\right)^{n}$ by $\left(\Omega,v\right)$. A tubular neighborhood of $M^{n}$ is a diffeomorphic image of $NM^{n}\to\operatorname{SE}\left(m\right)^{n}$,
\begin{equation}
\left(\left(R,p\right),\left(\Omega,v\right)\right)\mapsto\left(R\operatorname{exp}\left(R^{\top}\Omega\right),p+v\right),
\end{equation}
where $\operatorname{exp}:\mathfrak{so}\left(m\right)^{n}\to\operatorname{SO}\left(m\right)^{n}$ denotes the exponential map. The retraction from the tubular neighborhood onto $M^{n}$ is then given by 
\begin{equation}
 r:\left(R\operatorname{exp}\left(R^{\top}\Omega\right),p+v\right)\mapsto\left(R,p\right).
\end{equation}
Our function $\phi:M^{n}\to\mathbb{R}$ is still defined by \eqref{eq:phidef}, with $d\left(x_{i},x_{j}\right)$ being the length of the shortest curve (in $M$) joining $x_{i}$ and $x_{j}$. Thus, $\operatorname{grad}\phi$ takes elements of $M^{n}$ to tangent vectors thereof. Finally, thinking of $x=\left(R,p\right)$ in its homogeneous representation, let $x^{-1}$ denote the inverse element $\left(R^{\top},-R^{\top}p\right)$. Then, instead of \eqref{eq:ODE} we consider
\begin{equation}
 \dot{x}=x\operatorname{log}{\textstyle\left(x^{-1}\hspace{1pt}r\left(x\right)\right)}+xr\left(x\right)^{-1}\operatorname{grad}\phi\left(r\left(x\right)\right) \label{eq:SEODE}
\end{equation}
where $\operatorname{log}:\operatorname{SE}\left(m\right)^{n}\to\mathfrak{se}\left(m\right)^{n}$ is the logarithmic map and $xr\left(x\right)^{-1}\operatorname{grad}\phi\left(r\left(x\right)\right)$ is just the parallel transport of $\operatorname{grad}\phi\left(r\left(x\right)\right)$ from $T_{r\left(x\right)}\operatorname{SE}\left(m\right)^{n}$ to $T_{x}\operatorname{SE}\left(m\right)^{n}$.

\begin{theorem}
Let $X$ be a superlevel set of $\phi$ on which $\phi$ is regular away from the maximizers $X^{\ast}$. These maximizers constitute an asymptotically stable set of equilibria of \eqref{eq:SEODE} and $r^{-1}\left(X\right)$ is a subset of their region of asymptotic stability.
\end{theorem}
\begin{proof}
 The proof of Theorem \ref{thm:main} carries through except that $x-r\left(x\right)$ must be replaced by $\operatorname{log}\left(x^{-1}r\left(x\right)\right)$ in the second part of the proof (cf. \cite[proof of Theorem 1]{Montenbruckb}). 
\end{proof}

\section{Tutorial Examples: Circle and Sphere \\ in the Special Euclidean Groups}\label{sec:circlsphEucl}
The terminology we had to set up in order to work on the special Euclidean group became quite involved. It is instructive to see how the proposed differential equation \eqref{eq:SEODE} reads for a particular choice of $M$. In our next example, we thus explicitly compute the right-hand side of \eqref{eq:SEODE} for a formation {that} should be relevant in applications.
\begin{example}
 Let $m=2$, i.e., consider the special Euclidean group $\operatorname{SE}\left(2\right)$. A formation which should be of practical interest is to have $n$ agents arrange at equal distance around a target, say the origin, and face the target, with some device mounted along a body-fixed axis, say $e_{1}$, the first vector of the standard basis of $\mathbb{R}^{2}$. Letting $p_{i}$ and $R_{i}$ denote the position (in $\mathbb{R}^{2}$) and the orientation (in $\operatorname{SO}\left(2\right)$), respectively, of the $i$th system, then this goal is formalized by requiring that
 \begin{equation}
  p_{i}+R_{i}e_{1}=0, \label{eq:faceorig}
 \end{equation}
shall asymptotically hold for the $i$th system, if possible in a stable fashion. This being said, our target manifold $M$ is constituted by all points in $\operatorname{SE}\left(2\right)$ for which \eqref{eq:faceorig} holds. It is now a convenient fact that $\left\|p_{i}\right\|=1$ and $R_{i}=\Omega p_{i}e_{2}^{\top}-p_{i}e_{1}^{\top}$ is equivalent to $\left(R_{i},p_{i}\right)\in M$, where $e_{2}$ is the second vector of the standard basis of $\mathbb{R}^{2}$. This reveals that $M$ is just the circle, embedded in $\operatorname{SE}\left(2\right)$. Thus, here, we may still employ the differential equations \eqref{eq:circODE} to control the positions and the differential equations
\begin{equation}
 \dot{R}_{i}=R_{i}\operatorname{log}\left(\frac{1}{\left\|p_{i}\right\|}R_{i}^{\top}\left(\Omega p_{i}e_{2}^{\top}-p_{i}e_{1}^{\top}\right)\right) \label{eq:SE2ODE}
\end{equation}
to govern the orientations, wherein $\operatorname{log}:\operatorname{SO}\left(2\right)\to\mathfrak{so}\left(2\right)$ is the logarithmic map. Similarly, we could opt to merely steer $p_{i}$ towards $-R_{i}e_{1}$ and then balance the orientations only. The distinction between these two approaches is that we only have to communicate positions in the former case whilst only having to communicate orientations in the latter. We now consider $n=8$ systems coupled through the unweighted cycle graph $C_{8}$. With this choice of graph, we solved \eqref{eq:circODE} and \eqref{eq:SE2ODE} numerically for some initial condition; the numerical solutions are plotted in Fig. \ref{fig:SEfig1}. Therein, the initial condition is indicated by blue circles with arrows (\raisebox{1pt}{\tikz{\filldraw[blue] (0,0) circle (.06);\draw[->,blue,line width=.8] (0,0) -- (.35,0);}})  and the limiting point is marked by red circles with arrows (\raisebox{1pt}{\tikz{\filldraw[red] (0,0) circle (.06);\draw[->,red,line width=.8] (0,0) -- (.35,0);}}). The arrows are obtained from multiplying the attitudes $R_{i}$ with $e_{1}$. We find that the positions approach an evenly spaced configuration on the circle while facing the origin, as desired. 

If, on the other hand, one wishes that the orientations of our systems shall asymptotically point outwards, then it is sufficient to multiply the argument of the logarithmic map in \eqref{eq:SE2ODE} with $-1$. The resulting numerical solution is depicted in Fig. \ref{fig:SEfig2}. If one, instead, inserts $\Omega$ in between $R_{i}^{\top}$ and $\Omega p_{i}e_{2}^{\top}-p_{i}e_{1}^{\top}$ in that logarithm, then one makes $M$ the submanifold of $\operatorname{SE}\left(2\right)$ composed of points on the circle and orientations aligned with the tangent spaces of the circle, numerical solutions of the resulting differential equation being depicted in Fig. \ref{fig:SEfig2}.

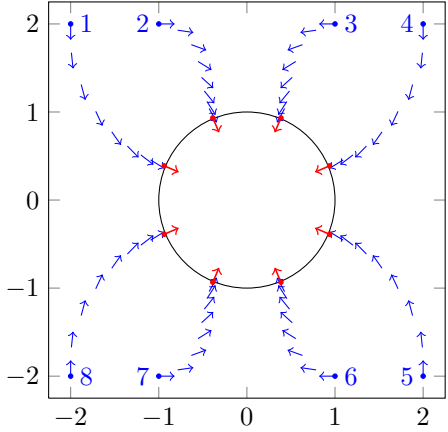
\begin{figure}
\centering
%
%
\begin{tikzpicture}[arrow1/.style={->,color=blue,solid},arrow2/.style={->,color=red,solid,line width=.6}]

\begin{axis}[%
width=15em,
height=15em,
scale only axis,
every outer x axis line/.append style={black},
every x tick label/.append style={font=\color{black}},
xmin=-2.25,
xmax=2.25,
every outer y axis line/.append style={black},
every y tick label/.append style={font=\color{black}},
ymin=-2.25,
ymax=2.25,
]
\draw (axis cs:0,0) circle (100);
\filldraw[blue] (axis cs:-2,2) circle (2.5) node[anchor=west] {$1$};
\filldraw[blue] (axis cs:-1,2) circle (2.5) node[anchor=east] {$2$};
\filldraw[blue] (axis cs:1,2) circle (2.5) node[anchor=west] {$3$};
\filldraw[blue] (axis cs:2,2) circle (2.5) node[anchor=east] {$4$};
\filldraw[blue] (axis cs:2,-2) circle (2.5) node[anchor=east] {$5$};
\filldraw[blue] (axis cs:1,-2) circle (2.5) node[anchor=west] {$6$};
\filldraw[blue] (axis cs:-1,-2) circle (2.5) node[anchor=east] {$7$};
\filldraw[blue] (axis cs:-2,-2) circle (2.5) node[anchor=west] {$8$};
\addplot [arrow1] coordinates{(-2,2) (-2,1.82)};
\addplot [arrow1] coordinates{(-1,2) (-0.82,2)};
\addplot [arrow1] coordinates{(1,2) (0.82,2)};
\addplot [arrow1] coordinates{(2,2) (2,1.82)};
\addplot [arrow1] coordinates{(2,-2) (2,-1.82)};
\addplot [arrow1] coordinates{(1,-2) (0.82,-2)};
\addplot [arrow1] coordinates{(-1,-2) (-0.82,-2)};
\addplot [arrow1] coordinates{(-2,-2) (-2,-1.82)};
\addplot [arrow1] coordinates{(-1.99149197440145,1.6736463287474) (-1.9728995454523,1.49460916485296)};
\addplot [arrow1] coordinates{(-0.78538043780709,1.92878322293415) (-0.607219456254783,1.90311932025201)};
\addplot [arrow1] coordinates{(0.78538043780709,1.92878322293415) (0.607219456254783,1.90311932025201)};
\addplot [arrow1] coordinates{(1.99149197440145,1.6736463287474) (1.9728995454523,1.49460916485296)};
\addplot [arrow1] coordinates{(1.99149197440145,-1.6736463287474) (1.9728995454523,-1.49460916485296)};
\addplot [arrow1] coordinates{(0.78538043780709,-1.92878322293415) (0.607219456254783,-1.90311932025201)};
\addplot [arrow1] coordinates{(-0.78538043780709,-1.92878322293415) (-0.607219456254783,-1.90311932025201)};
\addplot [arrow1] coordinates{(-1.99149197440145,-1.6736463287474) (-1.9728995454523,-1.49460916485296)};
\addplot [arrow1] coordinates{(-1.88020902476172,1.31217851377822) (-1.83375596782875,1.13827607007881)};
\addplot [arrow1] coordinates{(-0.630692826208462,1.76460132093157) (-0.46153896096269,1.70306371009893)};
\addplot [arrow1] coordinates{(0.630692826208462,1.76460132093157) (0.46153896096269,1.70306371009893)};
\addplot [arrow1] coordinates{(1.88020902476172,1.31217851377822) (1.83375596782875,1.13827607007881)};
\addplot [arrow1] coordinates{(1.88020902476172,-1.31217851377822) (1.83375596782875,-1.13827607007881)};
\addplot [arrow1] coordinates{(0.630692826208462,-1.76460132093157) (0.46153896096269,-1.70306371009893)};
\addplot [arrow1] coordinates{(-0.630692826208462,-1.76460132093157) (-0.46153896096269,-1.70306371009893)};
\addplot [arrow1] coordinates{(-1.88020902476172,-1.31217851377822) (-1.83375596782875,-1.13827607007881)};
\addplot [arrow1] coordinates{(-1.71510430371516,1.00759373432504) (-1.63914556803415,0.844406170788562)};
\addplot [arrow1] coordinates{(-0.554693903645049,1.57374057421742) (-0.401842305142334,1.47868129599742)};
\addplot [arrow1] coordinates{(0.554693903645049,1.57374057421742) (0.401842305142334,1.47868129599742)};
\addplot [arrow1] coordinates{(1.71510430371516,1.00759373432504) (1.63914556803416,0.844406170788562)};
\addplot [arrow1] coordinates{(1.71510430371516,-1.00759373432504) (1.63914556803415,-0.844406170788562)};
\addplot [arrow1] coordinates{(0.554693903645049,-1.57374057421742) (0.401842305142334,-1.47868129599742)};
\addplot [arrow1] coordinates{(-0.554693903645049,-1.57374057421742) (-0.401842305142334,-1.47868129599742)};
\addplot [arrow1] coordinates{(-1.71510430371516,-1.00759373432504) (-1.63914556803416,-0.844406170788562)};
\addplot [arrow1] coordinates{(-1.53326223453923,0.774622321980023) (-1.42998036172413,0.627202157512445)};
\addplot [arrow1] coordinates{(-0.514479254524959,1.39324916311317) (-0.381518869500383,1.27191748843007)};
\addplot [arrow1] coordinates{(0.514479254524959,1.39324916311317) (0.381518869500383,1.27191748843007)};
\addplot [arrow1] coordinates{(1.53326223453923,0.774622321980023) (1.42998036172413,0.627202157512445)};
\addplot [arrow1] coordinates{(1.53326223453923,-0.774622321980023) (1.42998036172413,-0.627202157512445)};
\addplot [arrow1] coordinates{(0.514479254524959,-1.39324916311317) (0.381518869500383,-1.27191748843007)};
\addplot [arrow1] coordinates{(-0.514479254524959,-1.39324916311317) (-0.381518869500383,-1.27191748843007)};
\addplot [arrow1] coordinates{(-1.53326223453923,-0.774622321980023) (-1.42998036172413,-0.627202157512445)};
\addplot [arrow1] coordinates{(-1.35613851337563,0.612843074622032) (-1.22997501823824,0.484458553835651)};
\addplot [arrow1] coordinates{(-0.483614253776805,1.23894544373018) (-0.370200105151036,1.09917041276565)};
\addplot [arrow1] coordinates{(0.483614253776805,1.23894544373018) (0.370200105151036,1.09917041276565)};
\addplot [arrow1] coordinates{(1.35613851337563,0.612843074622032) (1.22997501823824,0.484458553835651)};
\addplot [arrow1] coordinates{(1.35613851337563,-0.612843074622032) (1.22997501823824,-0.484458553835651)};
\addplot [arrow1] coordinates{(0.483614253776805,-1.23894544373018) (0.370200105151036,-1.09917041276565)};
\addplot [arrow1] coordinates{(-0.483614253776805,-1.23894544373018) (-0.370200105151036,-1.09917041276565)};
\addplot [arrow1] coordinates{(-1.35613851337563,-0.612843074622032) (-1.22997501823824,-0.484458553835651)};
\addplot [arrow1] coordinates{(-1.20342821014629,0.512252001914143) (-1.06041293319551,0.402951858555651)};
\addplot [arrow1] coordinates{(-0.453789544180116,1.11968388428623) (-0.356619557894619,0.968165523167449)};
\addplot [arrow1] coordinates{(0.453789544180117,1.11968388428623) (0.356619557894619,0.968165523167449)};
\addplot [arrow1] coordinates{(1.20342821014629,0.512252001914143) (1.06041293319551,0.402951858555651)};
\addplot [arrow1] coordinates{(1.20342821014629,-0.512252001914143) (1.06041293319551,-0.402951858555651)};
\addplot [arrow1] coordinates{(0.453789544180116,-1.11968388428623) (0.356619557894619,-0.968165523167449)};
\addplot [arrow1] coordinates{(-0.453789544180116,-1.11968388428623) (-0.356619557894619,-0.968165523167449)};
\addplot [arrow1] coordinates{(-1.20342821014629,-0.512252001914143) (-1.06041293319551,-0.402951858555651)};
\addplot [arrow1] coordinates{(-1.09276039980103,0.455541141959839) (-0.939296505410684,0.36147353728108)};
\addplot [arrow1] coordinates{(-0.428107561620019,1.03963055580648) (-0.342269176702439,0.881417222361374)};
\addplot [arrow1] coordinates{(0.428107561620019,1.03963055580648) (0.342269176702439,0.881417222361374)};
\addplot [arrow1] coordinates{(1.09276039980103,0.455541141959839) (0.939296505410685,0.361473537281079)};
\addplot [arrow1] coordinates{(1.09276039980103,-0.455541141959839) (0.939296505410685,-0.361473537281079)};
\addplot [arrow1] coordinates{(0.428107561620019,-1.03963055580648) (0.342269176702439,-0.881417222361374)};
\addplot [arrow1] coordinates{(-0.428107561620019,-1.03963055580648) (-0.342269176702439,-0.881417222361374)};
\addplot [arrow1] coordinates{(-1.09276039980103,-0.455541141959839) (-0.939296505410684,-0.361473537281079)};
\addplot [arrow2] coordinates{(-0.938385213784188,0.38868173731736) (-0.773033640278409,0.317557495437022)};
\addplot [arrow2] coordinates{(-0.386733226799246,0.93368795017669) (-0.316386456649477,0.76800492692673)};
\addplot [arrow2] coordinates{(0.386733226799246,0.93368795017669) (0.316386456649477,0.76800492692673)};
\addplot [arrow2] coordinates{(0.938385213784188,0.38868173731736) (0.773033640278409,0.317557495437022)};
\addplot [arrow2] coordinates{(0.938385213784188,-0.388681737317361) (0.773033640278409,-0.317557495437022)};
\addplot [arrow2] coordinates{(0.386733226799246,-0.93368795017669) (0.316386456649477,-0.76800492692673)};
\addplot [arrow2] coordinates{(-0.386733226799246,-0.93368795017669) (-0.316386456649477,-0.76800492692673)};
\addplot [arrow2] coordinates{(-0.938385213784188,-0.388681737317361) (-0.773033640278409,-0.317557495437022)};
\filldraw[red] (axis cs:-0.938385213784188,0.38868173731736) circle (2.5);
\filldraw[red] (axis cs:-0.386733226799246,0.93368795017669) circle (2.5);
\filldraw[red] (axis cs:0.386733226799246,0.93368795017669) circle (2.5);
\filldraw[red] (axis cs:0.938385213784188,0.38868173731736) circle (2.5);
\filldraw[red] (axis cs:0.938385213784188,-0.388681737317361) circle (2.5);
\filldraw[red] (axis cs:0.386733226799246,-0.93368795017669) circle (2.5);
\filldraw[red] (axis cs:-0.386733226799246,-0.93368795017669) circle (2.5);
\filldraw[red] (axis cs:-0.938385213784188,-0.388681737317361) circle (2.5);
\end{axis}
\end{tikzpicture}%
\hphantom{-1-1-!}
\caption{Numerical solution of \eqref{eq:circODE} and \eqref{eq:SE2ODE} for the cycle graph $C_{8}$.}\label{fig:SEfig1}
\end{figure}

\begin{figure}
\centering
\input{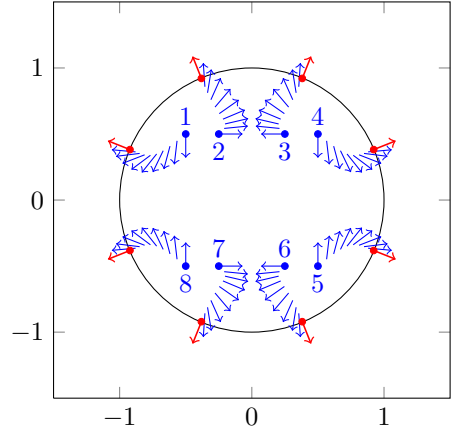}\hphantom{-1-1-!}
\caption{Numerical solution of \eqref{eq:circODE} and \eqref{eq:SE2ODE} for the cycle graph $C_{8}$, where $\Omega p_{i}e_{2}^{\top}-p_{i}e_{1}^{\top}$ is replaced by $p_{i}e_{1}^{\top}-\Omega p_{i}e_{2}^{\top}$ in \eqref{eq:SE2ODE}.}\label{fig:SEfig2}
\end{figure}

\begin{figure}
\centering
\input{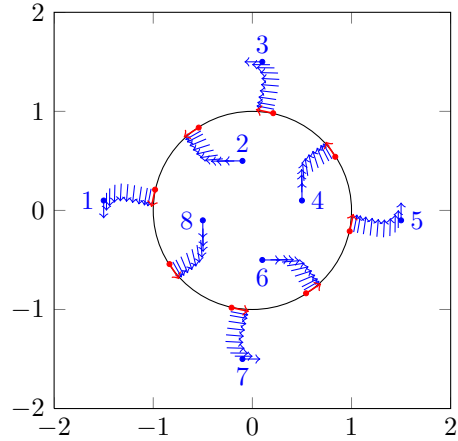}\hphantom{-1-1-!}
\caption{Numerical solution of \eqref{eq:circODE} and \eqref{eq:SE2ODE} for the cycle graph $C_{8}$, where $R_{i}^{\top}$ is replaced by $R_{i}^{\top}\Omega$ in \eqref{eq:SE2ODE}.}\label{fig:SEfig3}
\end{figure}

Now if $m$ was $3$, i.e., our systems moved in $\operatorname{SO}\left(3\right)$, and we would again ask for our systems to eventually face the origin, now with some device mounted along the body-fixed axis $e_{3}$ (the third vector of the standard basis of $\mathbb{R}^{3}$), then we could apply the injection \eqref{eq:sphereinj} to $-\left(1/\left\|p_{i}\right\|\right)p_{i}$ and replace $\left(1/\left\|p_{i}\right\|\right)\left(\Omega p_{i}e_{2}^{\top}-p_{i}e_{1}^{\top}\right)$ in \eqref{eq:SE2ODE} with the obtained rotation matrix in order to asymptotically stabilize the desired formation, where now $\operatorname{log}:\operatorname{SO}\left(3\right)\to\mathfrak{so}\left(3\right)$. For $n=5$ systems coupled through the unweighted complete graph $K_{5}$, we solved the resulting differential equation, together with \eqref{eq:sphereODE}, numerically for some initial condition, and plotted the numerical solution in Fig. \ref{fig:SE3fig}. Therein, again, the initial condition is indicated by blue circles with arrows (\raisebox{1pt}{\tikz{\filldraw[blue] (0,0) circle (.06);\draw[->,blue,line width=.8] (0,0) -- (.35,0);}})  and the limiting point is marked by red circles with arrows (\raisebox{1pt}{\tikz{\filldraw[red] (0,0) circle (.06);\draw[->,red,line width=.8] (0,0) -- (.35,0);}}). The arrows are obtained from multiplying the attitudes $R_{i}$ with $e_{3}$. We find that the positions approach the vertices of a triangular bipyramid whilst facing the origin, as desired.
 
\begin{figure}
\centering
\input{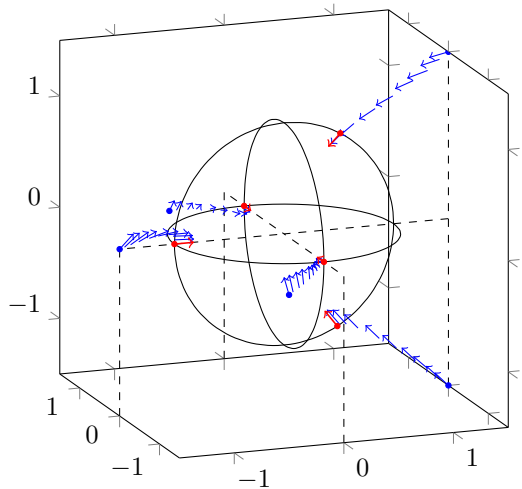}\hphantom{-1-1-!}
\caption{Numerical solution of \eqref{eq:sphereODE} and \eqref{eq:SE2ODE} for the complete graph $K_{5}$, where $\left(1/\left\|p_{i}\right\|\right)\left(\Omega p_{i}e_{2}^{\top}-p_{i}e_{1}^{\top}\right)$ is replaced by the rotation matrix obtained from applying \eqref{eq:sphereinj} to $-\left(1/\left\|p_{i}\right\|\right)p_{i}$ in \eqref{eq:SE2ODE}.}\label{fig:SE3fig}
\end{figure}

\end{example}

\section{Further Extensions}\label{sec:ext}

Some feasible extensions of the techniques proposed above will not be elaborated in detail herein. Yet, we briefly point out a few options to enhance our control \eqref{eq:ODE}.

{\emph{Backstepping:} If we could not influence the velocities of our systems directly, but could only actuate them on the acceleration level, i.e., if we had to control second-order dynamics, such as they arise in mechanical systems, then our approach is still applicable by virtue of the backstepping technique. More precisely, the system $\dot{x}=v$, $\dot{v}=u$ would asymptotically behave as \eqref{eq:ODE} if we applied the control $u=-v-\operatorname{J}_{f}\left(x\right)v-2f\left(x\right)$, wherein $f\left(x\right)$ denotes the right-hand side of \eqref{eq:ODE} and $\operatorname{J}_{f}:\mathbb{R}^{mn}\to\mathbb{R}^{mn\times mn}$ is the Jacobian of~$f$.}

\emph{Nearest neighbor communication:} Adaptive communication graphs could readily be incorporated into our setting. An example {that} should be of particular interest is to have the $i$th system communicate only with the systems whose positions are in a ball of certain radius, centered at $x_{i}$. If each system only communicated with its two closest neighbors, then we would precisely arrive at the cycle graph, which turned out to be suited for stabilization of evenly spaced configurations on the circle, as investigated in Example \ref{ex:circ}. More general, if each system only communicated with its $k$ closest neighbors, with $k$ an even positive number, we had a $k$-regular communication graph, leading to an evenly spaced configuration on the circle, as well, as discussed in Example \ref{ex:graphs}. 

\emph{Moving submanifolds:} Suppose we would not want all systems to eventually come to rest on $M$, but to have them collectively move in the desired formation. This could be formulated by translating our manifold $M$, i.e., to add the solution to some exosystem  $\dot{z}=f\left(z\right)$ to $M$ and hence replace $M$ by the affine translation $M+z$ throughout. Moreover, $f\left(z\right)$ would have to be added to the differential equation governing $x_{i}$, for all $i$, in order to guarantee asymptotic tracking of $M+z$.

\emph{Moving on the submanifold:} If we wanted our systems to move on the submanifold in the desired formation, such as depicted in Fig. \ref{fig:Efig_UTILITY}, then this could be incorporated into our setting by defining a vector field $f$ on $M$ and then adding $f\left(r\left(x_{i}\right)\right)$ to the differential equation governing $x_{i}$, for all $i$. This would cause all systems to move along the orbits of $f$ but to maintain the formation determined by the maxima of $\phi$ while doing so.

\emph{Formation shapes with singularities:} Triangular formations are relevant in applications, particularly in aviation \cite{Wang1999}, and have thus also been subject to theoretical studies \cite{Cao2011,Doerfler2010}. It would therefore be of interest to treat the case of $M$ being a polyhedron. In order to apply our methods, we would have to remove the singular points of $M$, for instance by locally smoothing them out, in order to recover the structure of a smooth manifold. This can indeed be done as locally as desired as long as the singular points are isolated (which is the case for polyhedra). In \cite{Montenbruck2015}, we illustrated this possibility on the very example of a triangle (cf. \cite[Fig. 5]{Montenbruck2015}).

\section{Conclusion}\label{sec:conc}\addtolength{\textheight}{-18.975cm}
We proposed a method for solving formation control problems. Our approach is based upon letting the shape of our formation be defined by some smooth compact submanifold. We then had our systems maximize a certain scalar field, defined on the submanifold, which itself has a rich history in the exact sciences (in which context the maximizers are called Fekete points). The control we proposed consists of a decentralized and a distributed component by construction. We demonstrated the flexibility of our approach on different examples and provided a graph-theoretical interpretation of the configurations that will eventually be attained through our control. Lastly, we equipped our control with the capability of taking into account formations {that} also specify the orientations of the systems and pointed out several further extensions.

\bibliographystyle{IEEEtran}
\bibliography{retrbalpap}

\end{document}